\documentclass[12pt]{article}

\usepackage{amsmath,amsfonts}
\usepackage{times}
\usepackage{graphicx}
\usepackage{multirow}
\usepackage{booktabs}
\usepackage[colorlinks=true, linkcolor=blue, urlcolor=blue, citecolor=blue]{hyperref}
\graphicspath{{fig/}{./}}
\DeclareMathOperator{\rank}{rank}
\DeclareMathOperator{\tr}{tr}
\DeclareMathOperator{\stt}{s.t.}

\DeclareMathOperator{\argmin}{argmin}
\DeclareMathOperator{\argmax}{argmax}
\usepackage{amsthm}
\theoremstyle{plain}
\newtheorem{theorem}{Theorem}
\newtheorem{remark}{Remark}
\newtheorem{corollary}{Corollary}
\newcommand{\F}{\mathrm{F}}

\newenvironment{sciabstract}{%
\begin{quote} \bf}
{\end{quote}}

\title{From Generality to Specificity: Prior-Driven Optimal Sparse Transformation in Compressed Sensing}

\author
{Zhihan Zhu,$^{1}$ Yanhao Zhang,$^{1}$ Yong Xia$^{1\ast}$\\
\\
\normalsize{$^{1}$School of Mathematical Sciences, Beihang University}\\
\normalsize{Beijing, 100191, China}\\
\\
\normalsize{$^\ast$Corresponding author. E-mail:  yxia@buaa.edu.cn}
}

\date{}

\begin{document} 


\baselineskip24pt


\maketitle


\begin{sciabstract}
  This paper introduces a new paradigm for sparse transformation: the Prior-to-Posterior Sparse Transform (POST) framework, designed to overcome long-standing limitation on generalization and specificity in classical sparse transforms for compressed sensing. POST systematically unifies the generalization capacity of any existing transform domains with the specificity of reference knowledge, enabling flexible adaptation to diverse signal characteristics. Within this framework, we derive an explicit sparse transform domain termed HOT, which adaptively handles both real and complex-valued signals. We theoretically establish HOT’s sparse representation properties under single and multiple reference settings, demonstrating its ability to preserve generalization while enhancing specificity even under weak reference information. Extensive experiments confirm that HOT delivers substantial meta-gains across audio sensing, 5G channel estimation, and image compression tasks, consistently boosting multiple compressed sensing algorithms under diverse multimodal settings with negligible computational overhead.
\end{sciabstract}


\section{Introduction}

Compressed Sensing (CS), renowned for its powerful theoretical foundations and versatile practice, has been a milestone in 21st-century scientific discoveries \cite{donoho2006compressed}, with applications in (medical) image processing \cite{wang2023memory}, computational biology \cite{dai2008compressive}, wireless communication \cite{nguyen2013compressive}, computer vision \cite{cevher2008compressive} and deep learning \cite{wu2019deep}. The core idea of compressed sensing is to recover a high-dimensional sparse vector by a small number of measurements:
\begin{equation}
	\mathbf{y}=\mathbf{\Phi}\mathbf{x}+\mathbf{n},\label{basis CS}
\end{equation}
where $\mathbf{x}\in \mathbb{C}^{N\times1}$ denotes the sparse vector of interest and $\mathbf{y}\in \mathbb{C}^{M\times1} (N\gg M)$ represents the response vector under the measurement matrix $\mathbf{\Phi}\in \mathbb{C}^{M\times N}$ with additive observation noise $\mathbf{n}\sim \mathcal{CN}(\mathbf{0},\sigma^2 \mathbf{I})$. In practice,  $\mathbf{x}$ itself is often not sparse but exhibits transform sparsity \cite{donoho2006compressed}, meaning it is compressible in some transformation domain $\mathbf{D}\in \mathbb{C}^{N\times N}$. Once $\mathbf{x}=\mathbf{D}\mathbf{w}$ where $\mathbf{w}$ exhibits sparsity, we could rewrite \eqref{basis CS} as follows:
\begin{equation}
	\mathbf{y}=\mathbf{\Phi}\mathbf{D}\mathbf{w}+\mathbf{n}.\label{transform CS}
\end{equation}
Compressed sensing theory states that if $\boldsymbol{\Phi}\mathbf{D}$ satisfies the Restricted Isometry Property (RIP) \cite{candes2006robust} and the sparsity of $\mathbf{w}$ meets certain criteria, then $\mathbf{w}$ (or equivalently $\mathbf{x}$) could be accurately reconstructed. Commonly-used compressed sensing algorithms include Orthogonal Matching Pursuit (OMP) \cite{tropp2007signal}, Basis Pursuit (BP) \cite{chen2001atomic}, Least Absolute Shrinkage and Selection Operator (LASSO) \cite{tibshirani1996regression}, etc.

Transform sparsity plays a crucial role in the successful application of compressed sensing to real-world data, but it often requires extensive effort to construct an appropriate transform domain $\mathbf{D}$. In fact, sparse transform domains had already emerged even before the advent of compressed sensing and were widely used in various image and video compression standards. Classic transform domains include Fourier, Wavelet, Cosine, etc.

The continuous form of the Fourier transform was developed by French mathematician Joseph Fourier in the early 19th century to solve the heat equation. Discrete Fourier Transform (DFT) was formally introduced in 1965 \cite{cooley1965algorithm}, marking a foundational milestone in signal processing and harmonic analysis. The mathematical expression of the Discrete Fourier Transform is as follows:
\begin{equation}
	\mathbf{D}_{DFT, (k,n)} = \frac{1}{\sqrt{N}}e^{-i\frac{2\pi}{N}(k-1)(n-1)},~k,n = 1, \dots, N.
\end{equation}
The Discrete Fourier Transform exhibits excellent energy concentration properties, making it widely applicable in compressed sensing. For example, wireless MIMO channels are compressible in Discrete Fourier Transform domain \cite{oppenheim1999discrete,selesnick2001discrete,bajwa2010compressed}, serving as the cornerstone of channel estimation in modern communications.

Discrete Wavelet Transform (DWT) was developed in 1980s, offering a multiscale alternative to traditional Fourier-based methods. Natural images exhibit sparsity in Discrete Wavelet Transform domain \cite{mallat1999wavelet}, which forms the basis of JPEG-2000 compression standard.

Discrete Cosine Transform (DCT) was proposed in 1974 \cite{ahmed1974discrete} and rapidly gained prominence for its exceptional energy compaction properties, laying the groundwork for multimedia compression standards. Discrete Cosine Transform forms the basis of JPEG compression standard in 1992, and MPEG-1/2/4 compression standard for video compression. Moreover, audio signals have sparse representation in Discrete Cosine Transform domain \cite{stankovic2018analysis}.

The standard forms of DFT, DWT, and DCT are all orthogonal transform domains, which not only are computationally efficient but also facilitate the satisfaction of RIP. Another significant advantage of these classical transform domains lies in their excellent generalization capability. Natural signals such as images, audios or wireless channels, regardless of their specific forms, are always compressible on such transform domains. However, the resulting transform sparsity may still be insufficient. According to compressed sensing theory, successful recovery can become more achievable as the sparsity level after transformation is smaller relative to the number of observations \cite{baraniuk2007compressive}. These classical transforms often struggle to achieve satisfactory sparsity, especially in scenarios with limited observations.

\begin{figure}[h]
	\centering
	\includegraphics[width=1.05\linewidth]{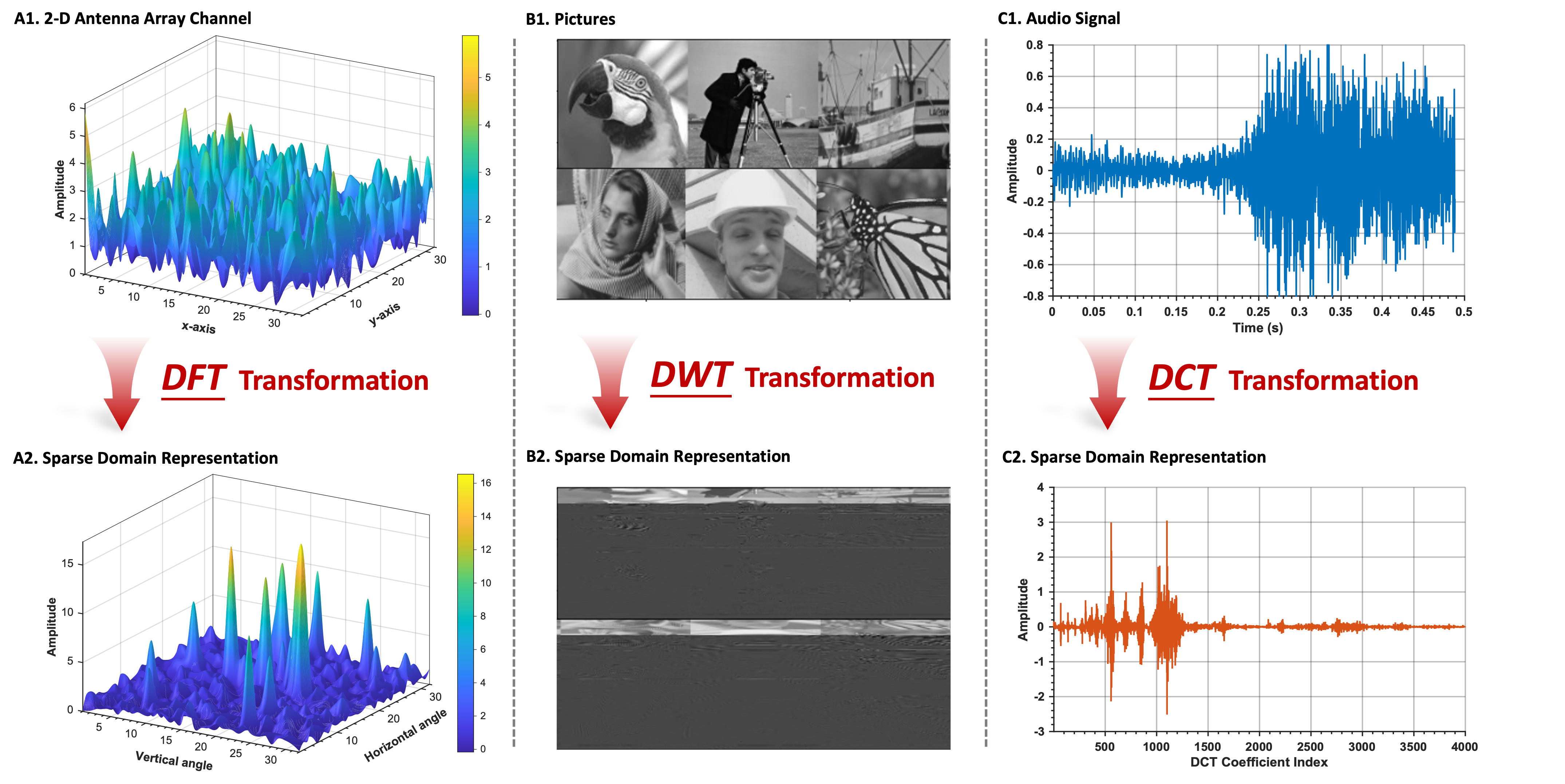}
	\vspace{-2mm}
	\caption{Wireless channel, images and audio signal exhibit sparsity on classic transform domains, yet the transform sparsity may still not be sufficient.}
	\label{fig1}
\end{figure}

We draw inspiration from the celebrated Monty Hall problem, a paradox in probability theory. When all three doors remain closed, each concealing either a goat or a car, we have no reference information—each option appears equally likely. Yet the moment one door is opened to reveal a goat, even though we still don’t know where the car is, the probabilities shift. The act of revealing—even partial information—changes the game. This subtle shift underscores a profound point: in many scientific problems, we do not operate in a realm of omniscience, but neither are we left entirely in the dark. Reference information, however limited, can be crucial.

This lesson extends to the design of classical transform domains. Their generality often comes at the expense of specificity—not because they are poorly constructed, but because they are rooted in the principle of symmetry, making “uniform guesses” in the absence of context. Take, for instance, the Discrete Fourier Transform (DFT) or the Discrete Cosine Transform (DCT). Both satisfy the elegant identity:
\begin{equation}
	\mathbf{D}_{DFT/DCT}^H (1, \dots, 1)^H = \sqrt{N} \mathbf{e}_1,
\end{equation}
where $\mathbf{e}_1$ denotes the first standard basis vector. These transforms are known for their remarkable energy compaction properties: when applied to a uniformly distributed signal, they concentrate all the energy into a single coefficient. In the absence of reference knowledge, such transforms offer optimal sparsity and efficiency—they are the rational default.

But this raises a compelling question: what if we could open a door for these transforms as well? In CS, although $\mathbf{x}$ cannot be completely known, it is always possible to obtain some reference knowledge $\mathbf{r}$ that has a varying degree of correlation with the real signal $\mathbf{x}$. Relying solely on classic transforms may waste this potential information, even when $\mathbf{r}$ is very inaccurate.

\begin{figure}[h]
	\centering
	\includegraphics[width=0.95\linewidth]{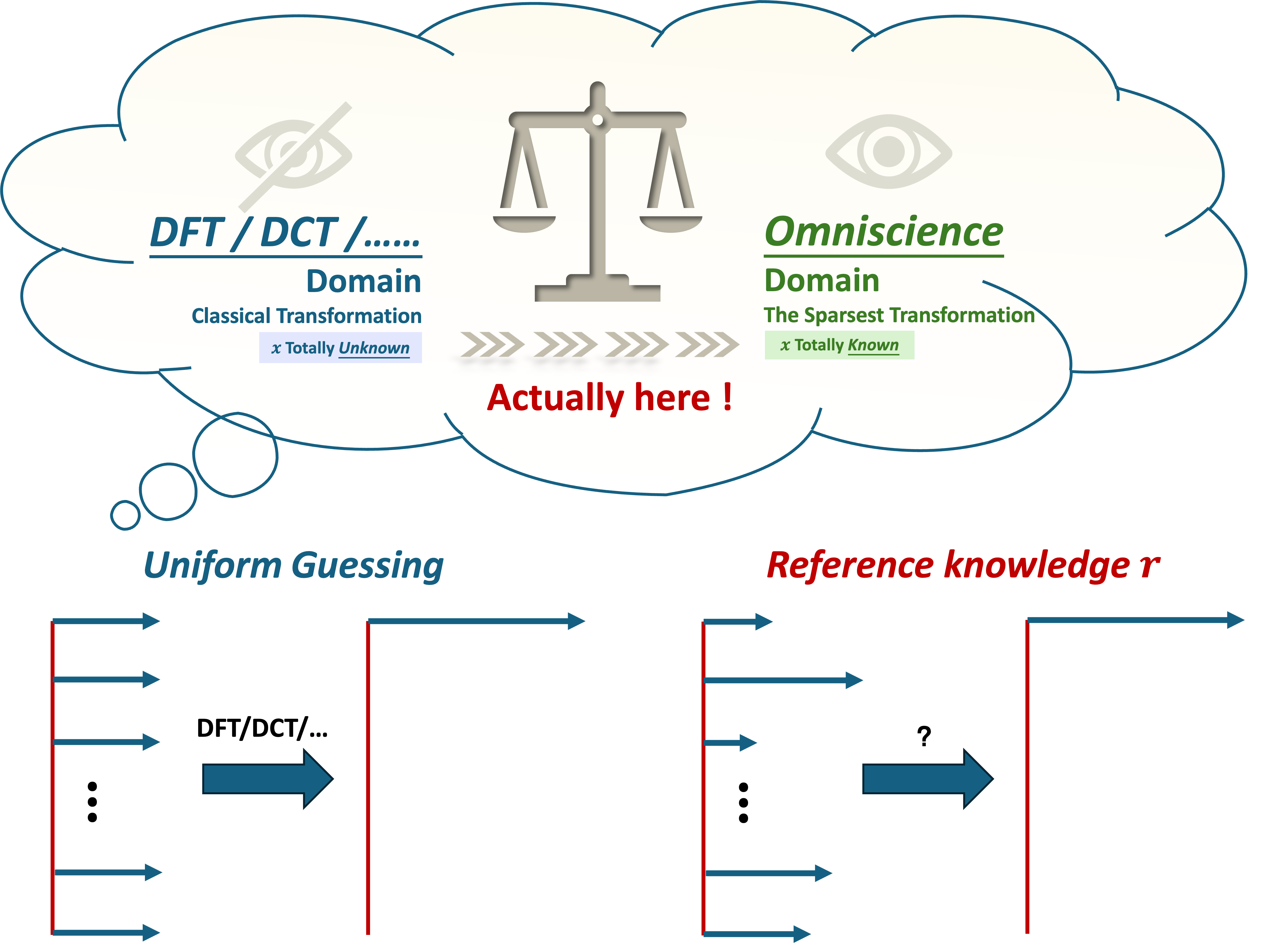}
	\vspace{-2mm}
	\caption{Open the door for classic transform and idea of POST.}
	\label{fig2}
\end{figure}

In light of the above, We hope the optimal transform domain satisfies the following:
\begin{itemize}
	\item Specificity: Such a transform substitute uniform guessing by reference knowledge
	to exhibit optimal sparsity.
	\item Generalizability: The transform still work even when reference knowledge is very inaccurate.
	\item Orthogonality: It is easy to compute and maintain basis incoherence.
\end{itemize}

In this paper, we propose a prior-to-posterior framework for domain transformation, termed Prior-driven Optimal Sparse Transformation (POST), which could be built upon any existing transformation and integrates the generalization of the prior domain with the specificity of reference knowledge. By using the rank function as objective within POST, we derive a novel explicit transformation domain: Householder Optimal Transformation (HOT). We demonstrate that HOT preserves the generalization ability comparable to the prior transform domain, while simultaneously exhibiting optimal specificity in sparse transformation with respect to reference information. Experimental results indicate that HOT provides significant meta-gains across various compressed sensing algorithms in multimodal data and diverse testing scenarios.

\section{Prior-driven Optimal Sparse Transformation}

In this section, we establish the prior-to-posterior domain transformation framework: Prior-driven Optimal Sparse Transformation (POST). Suppose, for the compressed sensing problem \eqref{basis CS}, there exists an orthogonal prior transform domain $\mathbf{D}_{prior}$ (such as DFT) with good generalization properties, and some reference knowledge $\mathbf{r}$ of the true signal $\mathbf{x}$. The POST framework can be represented as the following optimization problem:

\begin{equation}
	\begin{alignedat}{2}
		&\min_{\mathbf{D}_{post}} &&\quad \mathcal{L}\left(\mathbf{D}_{post},\mathbf{D}_{prior}\right)\\
		&\stt &&\quad \mathbf{D}_{post}^H\mathbf{r} = \alpha\mathbf{e}_j\\
		& && \quad\mathbf{D}_{post}^H\mathbf{D}_{post}=\mathbf{I}, 
		\label{POST}
	\end{alignedat}
\end{equation}
where objective function $\mathcal{L}$ quantifies the distance between posterior transform domain $\mathbf{D}_{post}$ and prior transform domain $\mathbf{D}_{prior}$, and $\mathbf{e}_j$ denotes the $j$-th standard basis vector with the index $j$ selected according to an arbitrary rule. In the POST framework, posterior transform domain $\mathbf{D}_{post}$ could share similar generalization properties to prior transform domain $\mathbf{D}_{prior}$ by minimization on the objective function. The first constraint guarantees that $\mathbf{D}_{post}$ exhibits optimal transform sparsity on the reference vector $\mathbf{r}$, while the second constraint ensures the conjugacy between the prior and posterior domains, i.e., orthogonality due to its significant computational efficiency.

The following theorem demonstrates that an appropriate selection on objective function could yield an explicit solution for the posterior transform domain.

\begin{theorem}\label{HOT}
	Choosing $\mathcal{L}\left(\mathbf{D}_{post},\mathbf{D}_{prior}\right)$ as $\rank\left(\mathbf{D}_{post}-\mathbf{D}_{prior}\right)$, then for $\forall \mathbf{r}\in\mathbb{C}^{N\times1}$, the global optimal solution of \eqref{POST} is formulated as 
	
    (a) $\mathbf{r} = \alpha\mathbf{D}_{prior, j}$:
	\begin{equation}
		\mathbf{D}_{post}= \mathbf{D}_{prior},
	\end{equation}
	
	(b) $\mathbf{r} \neq \alpha\mathbf{D}_{prior, j}$:
	\begin{equation}
		\mathbf{D}_{post}= \mathbf{D}_{prior}-\frac{2}{\vert|\mathbf{w}-\alpha\mathbf{e}_j\vert|^2}\left(\mathbf{r}-\alpha\mathbf{D}_{prior, j}\right)\left(\mathbf{w}-\alpha\mathbf{e}_j\right)^H,
		\label{solution1}
	\end{equation}
	where $\mathbf{w} =\mathbf{D}_{prior}^H \mathbf{r}$, $\vert \alpha\vert = \vert| \mathbf{r}\vert|_2$ with $\alpha\mathbf{w}^H\mathbf{e}_j\in\mathbb{R}$, and $\mathbf{D}_{prior, j}$ is the $j$-th column of $\mathbf{D}_{prior}$.  
\end{theorem}

\begin{proof}
	In case (a), prior transformation domain $\mathbf{D}_{prior}$ is directly the global optimal solution, hence only the non-trivial case in (b) need to be considered.
	
	In case (b), prior transformation domain $\mathbf{D}_{prior}$ itself is not a feasible solution, hence we start with minimal rank-one correction $\mathbf{D}_{post}=\mathbf{D}_{prior}+\mathbf{u}\mathbf{v}^H$, where $\mathbf{u},\mathbf{v}\in \mathbb{C}^{N\times1}$.
	
	By the first constraint in \eqref{POST}, we have
	\begin{equation}\label{1}
		\left(\mathbf{D}_{prior}^H+\mathbf{v}\mathbf{u}^H\right)\mathbf{r} = \alpha \mathbf{e}_j.
	\end{equation}
	Hence, 
	\begin{equation}\label{2}
		\mathbf{v}=k_1\left(\mathbf{D}_{prior}^H\mathbf{r}-\alpha \mathbf{e}_j \right),
	\end{equation}
	where $k_1 \in \mathbb{C}$. 
	
	Similarly, according to orthogonality of $\mathbf{D}_{post}$, we could rewrite \eqref{1} as 
	\begin{equation}\label{3}
		\mathbf{r} = \alpha \left(\mathbf{D}_{prior}+\mathbf{u}\mathbf{v}^H\right)\mathbf{e}_j.
	\end{equation}
	Hence, 
	\begin{equation}\label{4}
		\mathbf{u}=k_2\left(\mathbf{r}-\alpha\mathbf{D}_{prior, j} \right),
	\end{equation}
	where $k_2 \in \mathbb{C}$. 
	
	By \eqref{2} and \eqref{4}, we have
	\begin{equation}\label{5}
		\mathbf{D}_{post}=\mathbf{D}_{prior}+k\left(\mathbf{r}-\alpha\mathbf{D}_{prior, j} \right)\left(\mathbf{D}_{prior}^H\mathbf{r}-\alpha \mathbf{e}_j \right)^H,
	\end{equation}
	where $k \in \mathbb{C}$.
	Substituting \eqref{5} into \eqref{3}, we have
	\begin{align*}
		\mathbf{r} &= \alpha \left( \mathbf{D}_{prior}+k\left(\mathbf{r}-\alpha\mathbf{D}_{prior, j} \right)\left(\mathbf{D}_{prior}^H\mathbf{r}-\alpha \mathbf{e}_j \right)^H \right)\mathbf{e}_j\\
		&= \alpha \mathbf{D}_{prior,j} + \alpha k \left(\mathbf{r}-\alpha\mathbf{D}_{prior, j} \right)\left(\mathbf{r}^H\mathbf{D}_{prior, j}-\bar{\alpha} \right),
	\end{align*}
	where $\bar{\alpha}$ serves as the conjugate of $\alpha$. Hence, 
	\begin{equation*}
		\left(1-\alpha k \left(\mathbf{r}^H\mathbf{D}_{prior, j}-\bar{\alpha} \right) \right)\left(\mathbf{r}-\alpha\mathbf{D}_{prior, j} \right)=0.
	\end{equation*}
	Since $\mathbf{r} \neq \alpha\mathbf{D}_{prior, j}$ in case (b), we have
	\begin{equation*}
		1-\alpha k \left(\mathbf{r}^H\mathbf{D}_{prior, j}-\bar{\alpha} \right) =0.
	\end{equation*}
	 Hence, 
	 \begin{equation}
	 	k=\frac{1}{\alpha\mathbf{r}^H\mathbf{D}_{prior, j}-\vert| \mathbf{r}\vert|^2}.
	 \end{equation}
	 Since $\alpha\mathbf{r}^H\mathbf{D}_{prior, j} = \alpha\mathbf{w}^H\mathbf{e}_j\in\mathbb{R}$, $k\in\mathbb{R}$. Noticing that
	  \begin{equation*}
	 	\alpha\mathbf{r}^H\mathbf{D}_{prior, j}-\vert| \mathbf{r}\vert|^2=-\frac{1}{2}\vert| \mathbf{w}-\alpha\mathbf{e}_j\vert|^2,
	 \end{equation*}
	 we have
	 \begin{equation*}
	 	\mathbf{D}_{post}= \mathbf{D}_{prior}-\frac{2}{\vert|\mathbf{w}-\alpha\mathbf{e}_j\vert|^2}\left(\mathbf{r}-\alpha\mathbf{D}_{prior, j}\right)\left(\mathbf{w}-\alpha\mathbf{e}_j\right)^H.
	 \end{equation*}
	 The proof is complete.
\end{proof}

Theorem \ref{HOT} indicates that when selecting the rank as objective function in POST, the global optimal solution can be attained with at most a rank-one modification. Moreover, \eqref{solution1} can be transformed equivalently into an interesting form, as discussed in Corollary \ref{House}.

\begin{corollary}\label{House}
	Denote $\mathbf{v}=\frac{\mathbf{w}-\alpha\mathbf{e}_j}{\vert|\mathbf{w}-\alpha\mathbf{e}_j\vert|}$, $\mathbf{H}=\mathbf{I}-2\mathbf{v}\mathbf{v}^H$, then \eqref{solution1} becomes 
	\begin{equation}
		\mathbf{D}_{post}=\mathbf{D}_{prior}\left(\mathbf{I}-2\mathbf{v}\mathbf{v}^H\right)=\mathbf{D}_{prior}\mathbf{H}.
		\label{House_equ}
	\end{equation}
	The posterior transformation domain $\mathbf{D}_{post}$ corresponds to a Householder transformation $\mathbf{H}$ applied to the prior transformation domain $\mathbf{D}_{prior}$.
\end{corollary}
In accordance with Corollary \ref{House}, the posterior transform domain obtained from \eqref{solution1} or \eqref{House_equ} is termed as Householder Optimal Transformation (HOT). The naive Householder transformation is a special case of HOT when $\mathbf{D}_{prior}$ is an identity matrix.


\begin{remark}
	The HOT framework achieves domain adaptability, ensuring consistency between transform domains. Specifically, if the prior transform domain $\mathbf{D}_{prior}$ and reference vector $\mathbf{r}$ are real-valued, the posterior transform domain $\mathbf{D}_{post}$ inherently retains real-valued properties, enabling seamless integration into real-number-based signal processing pipelines.
\end{remark}

\section{Theoretical Properties of HOT}

In this section, we demonstrate the theoretical properties of HOT. The following theorem guarantees the similar generalization properties of HOT to the prior transformation domain \(\mathbf{D}_{prior}\).

\begin{theorem}\label{generalization}
	
	The relative error and correlation between the posterior transform domain $\mathbf{D}_{post}$ and the prior transform domain $\mathbf{D}_{prior}$ satisfy:
	\begin{equation}
		\mathcal{E}(\mathbf{D}_{post}, \mathbf{D}_{prior}) \le \frac{2}{\sqrt{N}}, \quad \rho(\mathbf{D}_{post}, \mathbf{D}_{prior}) \geq 1-\frac{2}{N},
	\end{equation}
	where
	\begin{equation}
		\mathcal{E}(\mathbf{D}_{post}, \mathbf{D}_{prior})=\frac{\vert|\mathbf{D}_{post}-\mathbf{D}_{prior}\vert|_\F}{\vert|\mathbf{D}_{prior}\vert|_\F},
	\end{equation}
	and
	\begin{equation}
		\rho(\mathbf{D}_{post}, \mathbf{D}_{prior}) = \frac{1}{N} \sum_{j=1}^{N} \frac{|\mathbf{D}_{post, j}^H\mathbf{D}_{prior, j}|}{\vert|\mathbf{D}_{post, j}\vert|\vert|\mathbf{D}_{prior, j}\vert|}.
	\end{equation}
\end{theorem}
\begin{proof}
	In Case (a) of Theorem \ref{HOT}, the posterior transform domain $\mathbf{D}_{post}$ and prior transform domain $\mathbf{D}_{prior}$ achieve a relative error of 0 and a perfect correlation of 1. Therefore, only the non-trivial Case (b) requires further analysis.
	
	We begin by analyzing the relative error between the posterior transform domain $\mathbf{D}_{post}$ and the prior transform domain $\mathbf{D}_{prior}$ in Case (b). 
	\begin{align*}
		\vert|\mathbf{D}_{post}-\mathbf{D}_{prior}\vert|_\F^2&=	\vert|\mathbf{D}_{prior}\left(\mathbf{I}-2\mathbf{v}\mathbf{v}^H\right)-\mathbf{D}_{prior}\vert|_\F^2 ~~~~({\rm by~ Corollary~\ref{House}})\\
		&=4\vert|\mathbf{D}_{prior}\mathbf{v}\mathbf{v}^H\vert|_\F^2\\
		&=4\tr\left(\left(\mathbf{D}_{prior}\mathbf{v}\mathbf{v}^H \right)^H \left(\mathbf{D}_{prior}\mathbf{v}\mathbf{v}^H \right)\right)\\
		&=4\tr\left(\mathbf{v}\mathbf{v}^H \right) ~~~~({\rm since~ \mathbf{v}~ is~ a~ unit~ vector})\\
		&=4.
	\end{align*}
	Since $	\vert|\mathbf{D}_{prior}\vert|_\F^2=N$, 
	\begin{equation*}
		\mathcal{E}(\mathbf{D}_{post}, \mathbf{D}_{prior}) = \frac{2}{\sqrt{N}}.
	\end{equation*}
	Now we analyze the correlation between the posterior transform domain $\mathbf{D}_{post}$ and the prior transform domain $\mathbf{D}_{prior}$. Noticing that
	\begin{align*}
		\mathbf{D}_{post, j}^H\mathbf{D}_{prior, j} &= \left(\mathbf{D}_{post}^H\mathbf{D}_{prior}\right)_{jj}\\
		&= \left(\left(\mathbf{I}-2\mathbf{v}\mathbf{v}^H\right)\mathbf{D}_{prior}^H\mathbf{D}_{prior}\right)_{jj} ~~~~({\rm by~ Corollary ~\ref{House}})\\
		&= \left(\mathbf{I}-2\mathbf{v}\mathbf{v}^H\right)_{jj},
	\end{align*}
	$\mathbf{D}_{post, j}^H\mathbf{D}_{prior, j} \in \mathbb{R}$. Hence,
	\begin{align*}
		\frac{1}{N} \sum_{j=1}^{N} \frac{|\mathbf{D}_{post, j}^H\mathbf{D}_{prior, j}|}{\vert|\mathbf{D}_{post, j}\vert|\vert|\mathbf{D}_{prior, j}\vert|} &=
			\frac{1}{N} \sum_{j=1}^{N} |\mathbf{D}_{post, j}^H\mathbf{D}_{prior, j}| ~~~~({\rm since~ \vert|\mathbf{D}_{post, j}\vert|=\vert|\mathbf{D}_{prior, j}\vert|=1})\\
			& \geq \frac{1}{N} \sum_{j=1}^{N} \mathbf{D}_{post, j}^H\mathbf{D}_{prior, j}\\
			& = \frac{1}{N} \tr \left(\mathbf{D}_{post}^H\mathbf{D}_{prior}\right)\\
			& = \frac{1}{N} \tr \left(\left(\mathbf{I}-2\mathbf{v}\mathbf{v}^H\right)\mathbf{D}_{prior}^H\mathbf{D}_{prior}\right) ~~~~({\rm by~ Corollary ~\ref{House}})\\
			& = \frac{1}{N} \tr \left(\mathbf{I}-2\mathbf{v}\mathbf{v}^H\right)\\
			& = 1-\frac{2}{N}. ~~~~({\rm since~ \mathbf{v}~ is~ a~ unit~ vector})
	\end{align*}
	Hence, 
	\begin{equation*}
		\rho(\mathbf{D}_{post}, \mathbf{D}_{prior}) \geq 1-\frac{2}{N}.
	\end{equation*}
	The proof is complete.	
\end{proof}

\begin{remark}
	Theorem \ref{generalization} states that in HOT, the relative error between the posterior transform domain $\mathbf{D}_{post}$ and the prior transform domain $\mathbf{D}_{prior}$ asymptotically converges to 0, while their correlation approaches 1 as the dimension $N$ increases. This implies that the posterior domain $\mathbf{D}_{post}$ not only represents a minimal correction of $\mathbf{D}_{prior}$ in the rank sense, but also remains nearly indistinguishable from $\mathbf{D}_{prior}$ under the metrics of relative error and correlation. Consequently, the two domains share nearly identical generalization properties.
\end{remark}

\begin{remark}
	Since the naive Householder transformation can be regarded as the posterior transform domain when the prior transform domain is the identity matrix, it exhibits similar generalization properties to the identity matrix. This also explains why the naive Householder matrix has poor generalization performance, as will be thoroughly demonstrated in Figure \ref{fig4}.
\end{remark}

Define the energy concentration of a vector $\mathbf{a} \in \mathbb{C}^N$ as the ratio of the squared magnitude of its largest component to the total energy (squared $\ell_2$-norm) of the vector:  
\begin{equation*}
	\gamma(\mathbf{a}) = \frac{\max_{1 \leq i \leq N} |a_i|^2}{\|\mathbf{a}\|_2^2}.
\end{equation*}
This metric quantifies how concentrated the energy of $\mathbf{a}$ is in a single component, with $\gamma(\mathbf{a}) \in \left[\frac{1}{N}, 1\right]$. Higher values indicate greater sparsity.

The following theorem reveals the specificity inherent to the posterior transform domain $\mathbf{D}_{post}$ in HOT.

\begin{theorem}\label{specificity} 

	Let $\mathbf{x}$ denote the true signal with sparse representations $\mathbf{w}_{prior}$ and $\mathbf{w}_{post}$ in the prior and posterior transform domains, respectively, i.e.,
	\begin{equation*}
		\mathbf{x} = \mathbf{D}_{prior}\mathbf{w}_{prior} = \mathbf{D}_{post}\mathbf{w}_{post},
	\end{equation*}
	and the correlation between reference knowledge $\mathbf{r}$ and true signal $\mathbf{x}$ is defined as:
	\begin{equation*}
		\rho = \frac{|\mathbf{r}^H\mathbf{x}|}{\vert|\mathbf{r}\vert|\vert|\mathbf{x}\vert|}.
	\end{equation*}
	
	(a) If $\rho \geq \sqrt{\gamma(\mathbf{w}_{prior})}$, then
	\begin{equation}\label{3a}
		\gamma(\mathbf{w}_{post}) \geq \gamma(\mathbf{w}_{prior}).
	\end{equation}
	
	(b) Let $\mathbf{w} = \mathbf{D}_{prior}^H \mathbf{r}$ denote the reference knowledge on the prior transform domain. As long as $\mathbf{w}$ captures partial support information of \(\mathbf{w}_{\text{prior}}\), i.e., $\text{supp}(\mathbf{w}) \subseteq \text{supp}(\mathbf{w}_{prior})$, and $\mathbf{r}^H \mathbf{D}_{prior,j} \neq 0$, then
	\begin{equation}
		\vert|\mathbf{w}_{post}\vert|_0 \le \vert|\mathbf{w}_{prior}\vert|_0.
	\end{equation}
	
	(c) Define prior sparsity odd \footnote{Also known as numerical sparsity in \cite{lopes2013estimating}.} as
	\begin{equation}
		odd = \frac{\vert|\mathbf{w}_{prior}\vert|_1}{\vert|\mathbf{w}_{prior}\vert|_2},
	\end{equation}
	where $1 \le odd \le \sqrt{N}$. We have
	\begin{equation}\label{3b}
		\vert|\mathbf{w}_{post}\vert|_1 \le \vert|\mathbf{w}_{prior}\vert|_1.
	\end{equation}
	when
	
	(i) $1 \le odd \le \sqrt{N-1}$: 
	\begin{equation}\label{a1}
		\rho\geq \frac{odd+\sqrt{(N-1)(N-odd^2)}}{N}.
	\end{equation}
	
	(ii) $\sqrt{N-1} \le odd \le \sqrt{N}$:
	\begin{equation}\label{a2}
		\rho\geq \frac{odd+\sqrt{(N-1)(N-odd^2)}}{N}~~or~~\rho\le \frac{odd-\sqrt{(N-1)(N-odd^2)}}{N}.
	\end{equation}
\end{theorem}

\begin{proof}
	In Case (a) of Theorem \ref{HOT}, since the posterior transform domain $\mathbf{D}_{post}$ coincides with the prior transform domain $\mathbf{D}_{prior}$, Cases (a), (b), and (c) of Theorem \ref{specificity} become trivial. Consequently, it suffices to focus on the non-trivial scenario in Case (b) of Theorem \ref{HOT}.
	
	We begin the proof with Case (a) in Theorem \ref{specificity}. By the constraints in \eqref{POST}, we have
	\begin{equation*}
		\mathbf{r} = \alpha \mathbf{D}_{post,j}.
	\end{equation*}
	Hence, 
	\begin{equation*}
		\left|\mathbf{w}_{post,j}\right| = \left|\mathbf{D}_{post,j}^H\mathbf{x}\right| = \left|\frac{1}{\bar{\alpha}}\mathbf{r}^H\mathbf{x}\right| 
		= \frac{\left|\mathbf{r}^H\mathbf{x}\right|}{\vert| \mathbf{r}\vert|_2}
		= \rho\vert| \mathbf{x}\vert|_2,
	\end{equation*}
	where $\mathbf{w}_{post,j}$ denotes the $j$-th component of the posterior sparse representation vector $\mathbf{w}_{post}$. Hence,
	\begin{equation*}
		\gamma(\mathbf{w}_{post}) \geq \frac{|\mathbf{w}_{post,j}|^2}{\|\mathbf{w}_{post}\|_2^2} 
		= \frac{\rho^2\vert| \mathbf{x}\vert|_2^2}{\|\mathbf{w}_{post}\|_2^2}
		= \rho^2 
		\geq \gamma(\mathbf{w}_{prior}).
	\end{equation*}
	
	We now proceed to the proof of Case (b) in Theorem \ref{specificity}. By \eqref{solution1},
	\begin{align*}
		\mathbf{w}_{post} &= \mathbf{D}_{post}^H\mathbf{x}\\
		&= \left(\mathbf{D}_{prior}^H-\frac{2}{\vert|\mathbf{w}-\alpha\mathbf{e}_j\vert|^2}\left(\mathbf{w}-\alpha\mathbf{e}_j\right)\left(\mathbf{r}-\alpha\mathbf{D}_{prior, j}\right)^H\right)\mathbf{x}\\
		&= \mathbf{w}_{prior} - \frac{2\left(\mathbf{r}-\alpha\mathbf{D}_{prior, j}\right)^H\mathbf{x}}{\vert|\mathbf{w}-\alpha\mathbf{e}_j\vert|^2}\left(\mathbf{w}-\alpha\mathbf{e}_j\right).
	\end{align*}
	Since $\mathbf{r}^H \mathbf{D}_{prior,j} = \mathbf{w}^H\mathbf{e}_{j}\neq 0$, $j \in \text{supp}(\mathbf{w})$. Hence, $\text{supp}(\mathbf{w}-\alpha\mathbf{e}_j) \subseteq \text{supp}(\mathbf{w}) \subseteq \text{supp}(\mathbf{w}_{prior})$. 
	
	Hence, 
	\begin{equation*}
		\vert|\mathbf{w}_{post}\vert|_0 \le \vert|\mathbf{w}_{prior}\vert|_0.
	\end{equation*}
	
	Finally, we establish the proof for Case (c) of Theorem \ref{specificity}. Consider the orthogonal decomposition of $\mathbf{w}_{prior}$ with respect to $\mathbf{w}$, i.e.,
	\begin{equation}\label{ortho}
		\mathbf{w}_{prior} = \frac{\mathbf{w}^H\mathbf{w}_{prior}}{\mathbf{w}^H\mathbf{w}}\mathbf{w}+\mathbf{w}_{prior, \perp},
	\end{equation}
	where $\mathbf{w}_{prior, \perp}$ denote the orthogonal complement of $\mathbf{w}$, satisfying $\mathbf{w}_{prior, \perp}^H\mathbf{w}=0$. Noticing that
	\begin{align*}
		\vert|\mathbf{w}_{prior}\vert|_2^2 &= \left|\frac{\mathbf{w}^H\mathbf{w}_{prior}}{\mathbf{w}^H\mathbf{w}}\right|^2\vert|\mathbf{w}\vert|_2^2 + \vert|\mathbf{w}_{prior,\perp}\vert|_2^2\\
		&= \frac{\left|\mathbf{w}^H\mathbf{w}_{prior}\right|^2}{\vert|\mathbf{w}\vert|_2^2} + \vert|\mathbf{w}_{prior,\perp}\vert|_2^2\\
		&= \frac{\left|\mathbf{r}^H\mathbf{x}\right|^2}{\vert|\mathbf{r}\vert|_2^2} + \vert|\mathbf{w}_{prior,\perp}\vert|_2^2\\
		&= \rho^2 \vert|\mathbf{x}\vert|_2^2+ \vert|\mathbf{w}_{prior,\perp}\vert|_2^2\\
		&= \rho^2 \vert|\mathbf{w}_{prior}\vert|_2^2+ \vert|\mathbf{w}_{prior,\perp}\vert|_2^2,
	\end{align*}
	we have
	\begin{equation}\label{energy}
		\vert|\mathbf{w}_{prior,\perp}\vert|_2 = \sqrt{1-\rho^2}\vert|\mathbf{w}_{prior}\vert|_2.
	\end{equation}.
	
	By Corollary \ref{House},
	\begin{equation*}
		\mathbf{x} = \mathbf{D}_{prior}\mathbf{w}_{prior} = \mathbf{D}_{prior}\mathbf{H}\mathbf{H}\mathbf{w}_{prior} = \mathbf{D}_{post}\mathbf{H}\mathbf{w}_{prior} = \mathbf{D}_{post}\mathbf{w}_{post}.
	\end{equation*}
	Hence, $\mathbf{w}_{post}=\mathbf{H}\mathbf{w}_{prior}$.
	
	Therefore,
	\begin{align*}
		\vert|\mathbf{w}_{post}\vert|_1 &= \vert|\mathbf{H}\mathbf{w}_{prior}\vert|_1\\
		&=\vert|\frac{\mathbf{w}^H\mathbf{w}_{prior}}{\mathbf{w}^H\mathbf{w}}\mathbf{H}\mathbf{w}+\mathbf{H}\mathbf{w}_{prior, \perp}\vert|_1 ~~~~({\rm by~ \eqref{ortho}})\\
		&= \left|\frac{\mathbf{w}^H\mathbf{w}_{prior}}{\mathbf{w}^H\mathbf{w}}\alpha\mathbf{e}_j\right| + \vert|\mathbf{H}\mathbf{w}_{prior, \perp}\vert|_1 ~~~~({\rm by~ Corollary ~\ref{House}})\\
		&= \frac{\left|\mathbf{r}^H\mathbf{x}\right|}{\vert|\mathbf{r}\vert|_2} + \vert|\mathbf{H}\mathbf{w}_{prior, \perp}\vert|_1 \\
		&= \rho \vert|\mathbf{x}\vert|_2 + \vert|\mathbf{H}\mathbf{w}_{prior, \perp}\vert|_1 \\
		&\le \rho \vert|\mathbf{x}\vert|_2 + \sqrt{\left(1-\rho^2\right)\left(N-1\right)}\vert|\mathbf{w}_{prior}\vert|_2 ~~~~({\rm by~ Cauchy's ~Inequality ~and ~\eqref{energy}})\\
		&= \left(\rho + \sqrt{\left(1-\rho^2\right)\left(N-1\right)}\right)\vert|\mathbf{w}_{prior}\vert|_2.
	\end{align*}
	
	When 
	\begin{equation}\label{equa}
		\rho + \sqrt{\left(1-\rho^2\right)\left(N-1\right)} \le \frac{\vert|\mathbf{w}_{prior}\vert|_1}{\vert|\mathbf{w}_{prior}\vert|_2} = odd,
	\end{equation}
	we have $\vert|\mathbf{w}_{post}\vert|_1 \le \vert|\mathbf{w}_{prior}\vert|_1$. \eqref{equa} is equivalent to 
	\begin{equation}\label{equa1}
		N\rho^2 -2~odd~\rho + \left(odd^2-N+1\right) \geq 0, ~~0\le \rho \le 1.
	\end{equation}
	Solving inequality \eqref{equa1} yields the condition \eqref{a1} and \eqref{a2} on $\rho$. The proof is complete.
\end{proof}

\begin{remark}
	Theorem \ref{specificity} establishes that when the reference knowledge $\mathbf{r}$ captures partial information about the true signal $\mathbf{x}$, the sparsity of $\mathbf{x}$ on the posterior transform domain $\mathbf{D}_{post}$—quantified by energy concentration, $\ell_0$-norm, and $\ell_1$-norm—outperforms its sparsity on the prior transform domain $\mathbf{D}_{prior}$. This breakthrough enables compressed sensing algorithms to exceed the phase transition limit traditionally constrained by the prior domain, thereby enhancing signal recovery performance in practical applications.
\end{remark}

\begin{remark}
	Under the strict conditions stated in Theorem \ref{specificity} (a) and (c), inequalities \eqref{3a} and \eqref{3b} are also strictly satisfied, i.e., equality is excluded. This demonstrates that the posterior transform domain $\mathbf{D}_{post}$ exhibits enhanced specificity as the reference knowledge $\mathbf{r}$ becomes more precise, enabling more effective sparse recovery and signal processing.
\end{remark}

\begin{remark}
	In Theorem \ref{specificity} (c), the parameter $odd$ quantifies the sparsity of the representation $\mathbf{w}_{prior}$ on the prior transform domain $\mathbf{D}_{prior}$. A larger $odd$ corresponds to a less sparse $\mathbf{w}_{prior}$.  
	
	(i) For $ 1 \leq odd \leq \sqrt{N-1} $, a smaller $odd$ demands a higher correlation $\rho $ between the reference knowledge $\mathbf{r}$ and the true signal $\mathbf{x}$. Specifically, when $odd = 1$, $\mathbf{w}_{prior}$ achieves best sparsity (sparsity of 1), and the correlation $\rho$ must equal 1 to satisfy \eqref{3b}.  
	
	(ii) For $odd \geq \sqrt{N-1}$, even a small correlation $\rho$ between $\mathbf{r}$ and $\mathbf{x}$ suffices to satisfy \eqref{3b}. This indicates that the reference knowledge $\mathbf{r}$ can provide meaningful information regardless of whether $\rho$ is large or small. 
	
	(iii) When $odd = \sqrt{N}$, the representation $\mathbf{w}_{post}$ on the posterior transform domain $\mathbf{D}_{post}$ is guaranteed to be sparser than $\mathbf{w}_{\text{prior}}$, irrespective of the choice of reference knowledge $\mathbf{r}$.  
\end{remark}

According to Theorems \ref{generalization} and \ref{specificity}, the posterior transform domain in the HOT framework achieves both the generalization and specificity properties of sparse representation simultaneously. As illustrated in Figure \ref{fig3}, panels (a) and (b) show the original image and its representation in the DWT domain, respectively. We then construct two versions of HOT \footnote{The index $j$ is simply chosen as 1.} using different forms of reference information: the column-wise mean of the image and the leading left singular vector associated with its largest singular value. The resulting image representations in the corresponding HOT domains are shown in panels (d) and (e). For a more direct comparison, heatmaps of the DWT (serving as the prior transform domain) and HOT (posterior transform domain) are provided in panels (c) and (f) \footnote{Shown here is the HOT heatmap in (d). The corresponding heatmap in (e) exhibits a similar structure.}, respectively. Although the differences between the prior (DWT) and posterior (HOT) domains appear subtle, the energy compaction achieved in the HOT domain is markedly superior. Remarkably, even a coarse reference—such as the mean of all image columns—suffices to produce a transformation with significantly enhanced sparsity and energy concentration. This observation underscores a central insight: domain specificity need not come at the cost of generalization. Even minor, informed adjustments to a classical (prior) transform can yield substantial representational gains.

\begin{figure}[h]
	\centering
	\includegraphics[width=0.99\linewidth]{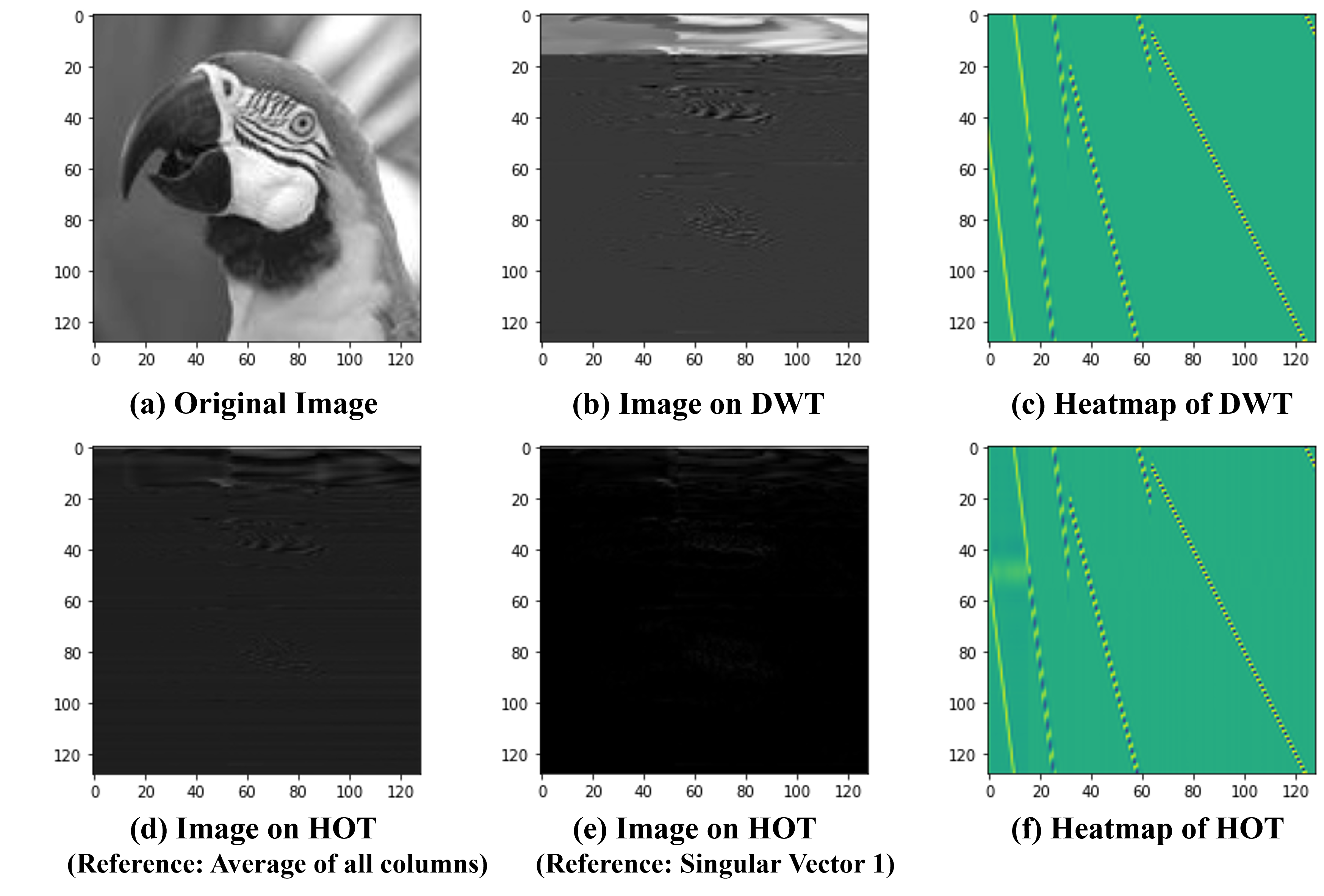}
	\vspace{-2mm}
	\caption{Specificity of HOT.}
	\label{fig3}
\end{figure}

Figure \ref{fig4} illustrates the generalization ability of HOT. In this example, we construct HOT using different prior transform domains while employing a randomly sampled Gaussian vector as reference information (very inaccurate case). Panels (b) and (c) show the resulting image representations in the posterior transform domains: HOT with DWT as the prior, and the other with the identity matrix as the prior—effectively corresponding to a naive Householder transform. The resulting posterior domains retain nearly identical representational structures to their respective priors. As indicated by Theorem \ref{generalization}, the posterior transform inherits the generalization characteristics of its prior, regardless of the reference information. This highlights the critical importance of choosing classical transform domains as priors within the POST framework: even when the reference knowledge is very inaccurate, HOT exhibit strong generalizability.

\begin{figure}[h]
	\centering
	\includegraphics[width=0.99\linewidth]{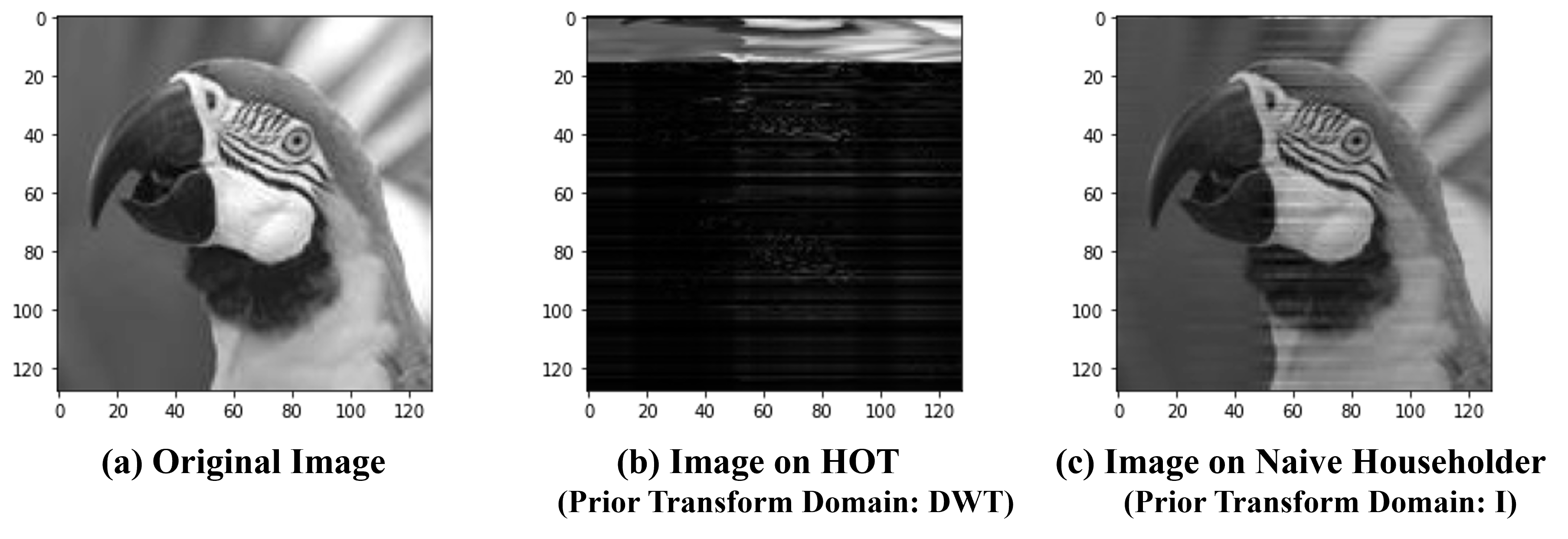}
	\vspace{-2mm}
	\caption{Generalizability of HOT.}
	\label{fig4}
\end{figure}

Another key advantage of HOT lies in the orthogonality of the posterior transform domain, which not only simplifies computational procedures but also guarantees optimal incoherence properties. This implies that as long as the measurement matrix $\mathbf{\Phi}$ satisfies the Restricted Isometry Property (RIP), all theoretical results for compressed sensing (with the optimal bound) established in \cite{randall2009sparse} are inherently satisfied.  

Specifically, when the measurement matrix $\mathbf{\Phi}$ is a random Gaussian matrix, the corresponding sensing matrix $\mathbf{\Phi}\mathbf{D}_{post}$ also retains the independent and identically distributed (i.i.d.) random Gaussian structure. This arises since the rows of $\mathbf{\Phi}\mathbf{D}_{post}$ remain independent, with each row following the distribution  $\mathcal{CN}(0, \mathbf{D}_{post}^H \mathbf{D}_{post}) = \mathcal{CN}(0, \mathbf{I})$. Consequently, all elements in $\mathbf{\Phi}\mathbf{D}_{post}$ are i.i.d. Gaussian. As a result, the sensing matrix $\mathbf{\Phi}\mathbf{D}_{post}$ satisfies the RIP with high probability, thereby perfectly preserving the theoretical guarantees of compressed sensing, as detailed in \cite{baraniuk2007compressive,candes2011compressed,randall2009sparse}.  

\section{Choice of $j$ in HOT}
In this section, we discuss the selection of the appropriate index $j$ within the HOT framework. The objective is to select the index $j$ such that the objective function $\mathcal{L}\left(\mathbf{D}_{post},\mathbf{D}_{prior}\right)$ and the relative error $\mathcal{E}(\mathbf{D}_{post}, \mathbf{D}_{prior})$ are jointly minimized. This leads to the following theorem.

\begin{theorem}\label{selection} 
	Let $\mathbf{D}_{prior}$ be the prior transform domain and $\mathbf{r}$ the reference knowledge. The optimal index $j^*$, which minimizes both the objective function $\mathcal{L}\left(\mathbf{D}_{post},\mathbf{D}_{prior}\right)$ and the relative error $\mathcal{E}(\mathbf{D}_{post}, \mathbf{D}_{prior})$, satisfies the following conditions:
	\begin{equation}\label{corr}
		j^* = \argmax_j \left|\mathbf{r}^H\mathbf{D}_{prior,j}\right|.
	\end{equation}
\end{theorem}

\begin{proof} 
	By Theorems \ref{HOT} and \ref{generalization}, when the posterior transform domain $\mathbf{D}_{post}$ is implemented as the rank-one correction \eqref{solution1} of the prior transform domain $\mathbf{D}_{prior}$, the relative error between the posterior and prior transform domains is $\frac{2}{\sqrt{N}}$. A potential scenario arises when there exists an index $j^*$ such that the reference vector $\mathbf{r} = \alpha \mathbf{D}_{prior,j^*}$. In this case, different choices of $j$ lead to distinct correction magnitudes. To jointly minimize the objective function and relative error, the optimal index $j^*$ should minimize the distance between reference knowledge $\mathbf{r}$ and basis in the prior transform domain $\mathbf{D}_{prior}$, which correspond to the solution of the following optimization problem:  
	\begin{equation}\label{l2}
		j^* = \argmin_j \vert|\mathbf{r}-\alpha\mathbf{D}_{prior,j}\vert|_2,
	\end{equation}
	where $\vert \alpha\vert = \vert| \mathbf{r}\vert|_2$ with $\alpha\mathbf{w}^H\mathbf{e}_j\in\mathbb{R}$ as in Theorem \ref{HOT}. Since \eqref{l2} is equivalent to \eqref{corr}, the proof is complete.
\end{proof}

\begin{remark}
	Criterion \eqref{corr} is equivalent to identifying the column of the prior transform domain $\mathbf{D}_{prior}$ that is maximally correlated with the reference knowledge $\mathbf{r}$. When Case (a) of Theorem \ref{HOT} holds, criterion \eqref{corr} can always identify the index $j^*$ such that $\mathbf{r} = \alpha\mathbf{D}_{prior,j^*}$, thereby eliminating the need for further correction to the prior transform domain.  
\end{remark}

\begin{remark}
	The index $j^*$ selected via criterion \eqref{corr} inherently satisfies one of the two conditions in Case (b) of Theorem \ref{specificity}, specifically $\mathbf{r}^H \mathbf{D}_{prior,j} \neq 0$. This demonstrates that when applying criterion \eqref{corr}, as long as $\mathbf{w}$ captures partial support information of \(\mathbf{w}_{\text{prior}}\), i.e., $\text{supp}(\mathbf{w}) \subseteq \text{supp}(\mathbf{w}_{prior})$, then
	\begin{equation}
		\vert|\mathbf{w}_{post}\vert|_0 \le \vert|\mathbf{w}_{prior}\vert|_0.
	\end{equation} 
\end{remark}

\section{HOT with Multiple Reference Knowledge}
In this section, we investigate HOT to incorporate multiple reference knowledge. In practical scenarios, it is often possible to acquire multiple reference knowledge about the target signals. For instance, when processing matrix or tensor data, we aim for the posterior transform domain to exhibit enhanced specificity across a collection of vectors. These reference knowledge components may include representative samples from the vector collection or subspace information characterizing their intrinsic structure. This motivates the development of HOT with multiple reference knowledge, which systematically integrates diverse information into the posterior transform domain for improved sparse representation.

Assume the multiple reference knowledge $\mathbf{r}_1, \dots, \mathbf{r}_K$ are linearly independent. We aim for the posterior transform domain $\mathbf{D}_{post}$ to maintain generalization akin to the prior transform domain $\mathbf{D}_{prior}$, while simultaneously exhibiting specificity for reference knowledge components. Therefore, the POST framework for multiple reference knowledge can be formulated as:  
\begin{equation}
	\begin{alignedat}{2}
		&\min_{\mathbf{D}_{post}} &&\quad \mathcal{L}\left(\mathbf{D}_{post},\mathbf{D}_{prior}\right)\\
		&\stt &&\quad \vert|\mathbf{D}_{post}^H\mathbf{r}_i\vert|_0 \le i, ~\forall i=1, \dots, K,\\
		& && \quad\mathbf{D}_{post}^H\mathbf{D}_{post}=\mathbf{I}.
		\label{POST1}
	\end{alignedat}
\end{equation}

This POST framework with multiple reference knowledge shares similarities with the formulation in \eqref{POST}, but introduces a critical distinction in its first constraint: we impose the sparsity condition $\|\mathbf{D}_{post}^H \mathbf{r}_i\|_0 \leq i$ for all $i = 1, \dots, K$. This constraint arises from the fact that for a reference matrix $\mathbf{R} = (\mathbf{r}_1, \dots, \mathbf{r}_K)$, the optimal sparsity of $\mathbf{D}_{post}^H \mathbf{R}$, under orthogonality constraints of $\mathbf{D}_{post}$, corresponds to an upper triangular matrix structure. Specifically, when the reference vectors $\{\mathbf{r}_i\}$ are mutually orthogonal, $\mathbf{D}_{post}^H \mathbf{R}$ reduces to a diagonal matrix, achieving the ideal sparsity level $\|\mathbf{D}_{post}^H \mathbf{r}_i\|_0 = 1$ for all $i$. However, such orthogonality represents a special case rather than the general scenario, necessitating the proposed relaxed constraint  $\|\mathbf{D}_{post}^H \mathbf{r}_i\|_0 \leq i$ to accommodate practical non-orthogonal reference configurations.

Similar to Theorem \ref{HOT}, an explicit solution for the posterior transform domain $\mathbf{D}_{post}$ can be derived by appropriately selecting the objective function. We formalize the following theorem.

\begin{theorem}
	Choosing $\mathcal{L}\left(\mathbf{D}_{post},\mathbf{D}_{prior}\right)$ as $\rank\left(\mathbf{D}_{post}-\mathbf{D}_{prior}\right)$, the global optimal solution of \eqref{POST1} corresponds to at most rank $K$ correction of the prior transform domain $\mathbf{D}_{prior}$, formulated as 
	\begin{equation}\label{HOT1}
		\mathbf{D}_{post} = \mathbf{D}_{prior}\mathbf{H}_1\cdots\mathbf{H}_K,
	\end{equation}
	where $\mathbf{H}_i = \mathbf{I} - 2\mathbf{v}_i\mathbf{v}_i^H$ is Householder matrix with $\mathbf{v}_i$ being constructible sequentially, satisfying $\vert|\mathbf{H}_i\cdots\mathbf{H}_1\mathbf{w}_i\vert|_0 \le i$, and $\mathbf{w}_i = \mathbf{D}_{prior}^H \mathbf{r}_i$ for all $i=1,\cdots,K$.
\end{theorem}

\begin{proof}
	In Theorem \ref{HOT}, it has been proven that when $K = 1$, the posterior transform domain $\mathbf{D}_{post}$ is at most rank-one correction of the prior transform domain $\mathbf{D}_{prior}$. Assuming that for the case of $K-1$ reference knowledge components, $\mathbf{D}_{post}$ is at most rank-$(K-1)$ correction of $\mathbf{D}_{prior}$, we now consider the case of $K$ reference knowledge components.  
	
	First, for the first $K-1$ reference knowledge components, the solution $\mathbf{D}_{post}$ satisfying \eqref{POST1} is at most rank-$(K-1)$ correction of $\mathbf{D}_{prior}$. The case of $K$ reference knowledge components is equivalent to adding a new constraint $\|\mathbf{D}_{post}^H \mathbf{r}_K\|_0 \leq K$ to \eqref{POST1}. To satisfy this new constraint, at most rank-one correction must be applied to the posterior transform domain for the first $K-1$ reference knowledge components. Therefore, for the case of $K$ reference knowledge components, the global optimal solution of \eqref{POST1} corresponds to at most rank $K$ correction of the prior transform domain $\mathbf{D}_{prior}$.  
	
	Now, we present the construction of the posterior transformation domain $\mathbf{D}_{post}$ for multiple reference knowledge.
	
	Suppose $\Omega = \{1,\cdots,N\}$ as the complete index set. For a vector $\mathbf{a} \in \mathbb{C}^{N \times 1}$ and an index set $S \subset \Omega$, define $\mathbf{a}_S \in \mathbb{C}^{N \times 1}$ such that $(\mathbf{a}_S)_j = a_j $ if $j \in S$, and $(\mathbf{a}_S)_j = 0$ if $j \in S^c$. $\mathbf{w}_i = \mathbf{D}_{prior}^H \mathbf{r}_i$ represents the reference knowledge on the prior transform domain.  
	  
	First, we construct a Householder matrix $\mathbf{H}_1$ as in Theorem \ref{HOT}, such that $\mathbf{H}_1 \mathbf{w}_1 = \alpha_1 \mathbf{e}_{j_1}$, where $\vert \alpha_1\vert = \vert| \mathbf{r}_1\vert|_2$ with $\alpha_1\mathbf{w}_1^H\mathbf{e}_{j_1}\in\mathbb{R}$. Now, for multiple reference vectors $\{\mathbf{w}_1, \mathbf{w}_2, \dots, \mathbf{w}_K\}$, we iteratively apply Householder transform:  
	
	- Align $\mathbf{w}_1$ to $\alpha_1 \mathbf{e}_{j_1}$ by $\mathbf{H}_1$.  
	
	- Update $\mathbf{w}_2$ to $\mathbf{H}_1\mathbf{w}_2$, then construct Householder transform $\mathbf{H}_2 = \mathbf{I} - 2\mathbf{v}_2\mathbf{v}_2^H$ such that 
	\begin{equation}
		\mathbf{v}_2 = \frac{(\mathbf{H}_1\mathbf{w}_2 - \alpha_2\mathbf{e}_{j_2})_{\Omega \setminus \{j_1\}}}{\vert|(\mathbf{H}_1\mathbf{w}_2 - \alpha_2\mathbf{e}_{j_2})_{\Omega \setminus \{j_1\}}\vert|_2},
	\end{equation} 	
	where $\vert \alpha_2\vert = \vert| (\mathbf{H}_1\mathbf{w}_2)_{\Omega \setminus \{j_1\}}\vert|_2$ with $\alpha_2\mathbf{w}_2^H\mathbf{H}_1^H\mathbf{e}_{j_2}\in\mathbb{R}$. Such an $\mathbf{H}_2$ satisfies:  
	(1) $ \mathbf{H}_2 \mathbf{H}_1 \mathbf{w}_1 = \alpha_1 \mathbf{e}_{j_1} $,  
	(2) $\text{supp}(\mathbf{H}_2 \mathbf{H}_1 \mathbf{w}_2) = \{j_1, j_2\}$.  
	
	- Update $\mathbf{w}_i$ to $\mathbf{H}_{i-1}\cdots\mathbf{H}_1\mathbf{w}_i$, then construct Householder transform $\mathbf{H}_i = \mathbf{I} - 2\mathbf{v}_i\mathbf{v}_i^H$ such that
	\begin{equation}\label{multiHOT}
		\mathbf{v}_i = \frac{(\mathbf{H}_{i-1}\cdots\mathbf{H}_1\mathbf{w}_i - \alpha_i\mathbf{e}_{j_i})_{\Omega \setminus \{j_1, j_2,\cdots j_{i-1}\}}}{\vert|(\mathbf{H}_{i-1}\cdots\mathbf{H}_1\mathbf{w}_i - \alpha_i\mathbf{e}_{j_i})_{\Omega \setminus \{j_1, j_2,\cdots j_{i-1}\}}\vert|_2},
	\end{equation} 	
	where $\vert \alpha_i\vert = \vert|(\mathbf{H}_{i-1}\cdots\mathbf{H}_1\mathbf{w}_i)_{\Omega \setminus \{j_1, j_2,\cdots j_{i-1}\}}\vert|_2$ with $\alpha_i\mathbf{w}_i^H\mathbf{H}_{1}^H\cdots\mathbf{H}_{i-1}^H\mathbf{e}_{j_i}\in\mathbb{R}$. Such an $\mathbf{H}_i$ satisfies: $\text{supp}(\mathbf{H}_i\cdots \mathbf{H}_1 \mathbf{w}_r) = \{j_1, j_2,\cdots j_r\}$ for all $r = 1, 2,\cdots i$. Repeat until all reference knowledge components are processed. 
	
	Finally, we obtain the posterior transform domain $\mathbf{D}_{post}$ as:  
	\begin{equation}
		\mathbf{D}_{post} = \mathbf{D}_{prior}\mathbf{H}_1\cdots\mathbf{H}_K,
	\end{equation}
	where $\mathbf{H}_i = \mathbf{I} - 2\mathbf{v}_i\mathbf{v}_i^H$ and $\mathbf{v}_i$ is define as \eqref{multiHOT}. From the construction process of $\mathbf{H}_i$, it is straightforward to observe that $\vert|\mathbf{D}_{post}^H\mathbf{r}_i\vert|_0 \le i, ~\forall i=1, \dots, K$. And due to the orthogonality of $\mathbf{H}_i$, we have $\mathbf{D}_{post}^H\mathbf{D}_{post}=\mathbf{I}$.
	
	We now prove that the posterior transform domain $\mathbf{D}_{post}$ is at most rank $K$ correction of the prior transform domain $\mathbf{D}_{prior}$.  When $K = 1$, Theorem \ref{HOT} shows that the posterior transformation domain $\mathbf{D}_{post}$ is at most rank-one correction of the prior transformation domain $\mathbf{D}_{prior}$.  
	
	Now, assume that for the case of $K-1$ reference knowledge components, the posterior transform domain \eqref{HOT1} is at most rank-$(K-1)$ modification of $\mathbf{D}_{prior}$. For the case of $K$ reference knowledge components, we have:
	\begin{align*}
		\mathbf{D}_{post} &= \mathbf{D}_{prior}\mathbf{H}_1\cdots\mathbf{H}_K\\
		&= \mathbf{D}_{prior}\mathbf{H}_1\cdots\mathbf{H}_{K-1}\left(\mathbf{I}-2\mathbf{v}_K\mathbf{v}_K^H\right)\\
		&= \underbrace{\mathbf{D}_{prior}\mathbf{H}_1\cdots\mathbf{H}_{K-1}}_{rank\text{-}(K-1) ~correction} - \underbrace{2\mathbf{D}_{prior}\mathbf{H}_1\cdots\mathbf{H}_{K-1}\mathbf{v}_K\mathbf{v}_K^H}_{rank~1},
	\end{align*}  
	which is at most rank $K$ correction of the prior transform domain $\mathbf{D}_{prior}$. The proof is complete.
\end{proof}

\begin{remark}
	The constraint $\|\mathbf{D}_{post}^H\mathbf{r}_i\|_0 \leq i $ in \eqref{POST1} indicates that in scenarios involving multiple reference signals, the posterior transform domain enhances the sparsity of earlier references more significantly. This allows more important reference information to be prioritized by positioning it earlier in the sequence. When the reference signals are mutually orthogonal, the posterior transform domain can achieve optimal sparsity for all references simultaneously, rendering the ordering of references no longer critical.
\end{remark}

\begin{remark}
	Equation \eqref{HOT1} can be interpreted as QR decomposition on the prior transformation domain $\mathbf{D}_{prior}$, implemented via a sequence of HOT to incorporate multi-reference knowledge.
\end{remark}

\begin{remark}\label{selection1}
	The selection of $j_1,\dots,j_K$ can be referred to Theorem \ref{selection}.
\end{remark}

In terms of HOT with multiple reference knowledge, We can similarly establish the generality theorem between the prior transform domain $\mathbf{D}_{prior}$ and posterior transform domain $\mathbf{D}_{post}$.

\begin{theorem}\label{generalization1}
	The relative error and correlation between the posterior transform domain $\mathbf{D}_{post}$ and the prior transform domain $\mathbf{D}_{prior}$ satisfy:
	\begin{equation}
		\mathcal{E}(\mathbf{D}_{post}, \mathbf{D}_{prior}) \le \frac{2\sqrt{K}}{\sqrt{N}}, \quad \rho(\mathbf{D}_{post}, \mathbf{D}_{prior}) \geq 1-\frac{2K}{N},
	\end{equation}
	where
	\begin{equation}
		\mathcal{E}(\mathbf{D}_{post}, \mathbf{D}_{prior})=\frac{\vert|\mathbf{D}_{post}-\mathbf{D}_{prior}\vert|_\F}{\vert|\mathbf{D}_{prior}\vert|_\F},
	\end{equation}
	and
	\begin{equation}
		\rho(\mathbf{D}_{post}, \mathbf{D}_{prior}) = \frac{1}{N} \sum_{j=1}^{N} \frac{|\mathbf{D}_{post, j}^H\mathbf{D}_{prior, j}|}{\vert|\mathbf{D}_{post, j}\vert|\vert|\mathbf{D}_{prior, j}\vert|}.
	\end{equation}
\end{theorem}

\begin{proof}
	Denote $\Re(z)$ as the real part of a complex number $z \in \mathbb{C}$. Noticing that
	\begin{align*}
		\Re\left(\tr \left(\mathbf{H}_K\cdots\mathbf{H}_1\right)\right) & = \Re\left(\tr \left(\left(\mathbf{I}-2\mathbf{v}_K\mathbf{v}_K^H\right)\mathbf{H}_{K-1}\cdots\mathbf{H}_1\right)\right)\\
		& = \Re\left(\tr \left(\mathbf{H}_{K-1}\cdots\mathbf{H}_1\right) - 2\tr\left(\mathbf{v}_K\mathbf{v}_K^H\mathbf{H}_{K-1}\cdots\mathbf{H}_1\right)\right)\\
		& = \Re\left(\tr \left(\mathbf{H}_{K-1}\cdots\mathbf{H}_1\right)\right) -  2\Re(\mathbf{v}_K^H\mathbf{H}_{K-1}\cdots\mathbf{H}_1\mathbf{v}_K)\\
		& \geq \Re\left(\tr \left(\mathbf{H}_{K-1}\cdots\mathbf{H}_1\right)\right) -  2 |\mathbf{v}_K^H\mathbf{H}_{K-1}\cdots\mathbf{H}_1\mathbf{v}_K|\\
		& \geq \Re\left(\tr \left(\mathbf{H}_{K-1}\cdots\mathbf{H}_1\right)\right) -  2 ~~~~({\rm by~ Cauchy's ~Inequality})
	\end{align*}
	and 
	\begin{align*}
		\Re(\tr\left(\mathbf{H}_1\right)) &= \Re(\tr\left(\mathbf{I}-2\mathbf{v}_1\mathbf{v}_1^H\right))\\
		&= N - 2,
	\end{align*}
	we have
	\begin{equation}\label{a3}
		\Re\left(\tr \left(\mathbf{H}_K\cdots\mathbf{H}_1\right)\right) \geq N - 2K.
	\end{equation}
	Hence,
	\begin{align*}
		\vert|\mathbf{D}_{post}-\mathbf{D}_{prior}\vert|_\F^2&=	\tr\left(\left(\mathbf{D}_{post}-\mathbf{D}_{prior}\right)^H\left(\mathbf{D}_{post}-\mathbf{D}_{prior}\right)\right)\\
		&= 2N - 2\Re\left(\tr\left(\mathbf{D}_{post}^H\mathbf{D}_{prior}\right)\right)\\
		&= 2N -  2\Re\left(\tr\left(\mathbf{H}_K\cdots\mathbf{H}_1\right)\right)~~~~({\rm by~ \eqref{HOT1}})\\
		&\le 2N - 2\left(N-2K\right)~~~~({\rm by~ \eqref{a3}})\\
		&= 4K.
	\end{align*}
	Since $	\vert|\mathbf{D}_{prior}\vert|_\F^2=N$, 
	\begin{equation*}
		\mathcal{E}(\mathbf{D}_{post}, \mathbf{D}_{prior}) = \frac{2\sqrt{K}}{\sqrt{N}}.
	\end{equation*}
	Now we analyze the correlation between the posterior transform domain $\mathbf{D}_{post}$ and the prior transform domain $\mathbf{D}_{prior}$. Noticing that
	\begin{align*}
		\frac{1}{N} \left|\tr \left(\mathbf{H}_K\cdots\mathbf{H}_1\right)\right| & = \frac{1}{N} \left|\tr \left(\left(\mathbf{I}-2\mathbf{v}_K\mathbf{v}_K^H\right)\mathbf{H}_{K-1}\cdots\mathbf{H}_1\right)\right|\\
		& = \frac{1}{N} \left|\tr \left(\mathbf{H}_{K-1}\cdots\mathbf{H}_1\right) - 2\tr\left(\mathbf{v}_K\mathbf{v}_K^H\mathbf{H}_{K-1}\cdots\mathbf{H}_1\right)\right|\\
		& \geq \frac{1}{N} \left|\tr \left(\mathbf{H}_{K-1}\cdots\mathbf{H}_1\right)\right| -  \frac{2}{N} |\mathbf{v}_K^H\mathbf{H}_{K-1}\cdots\mathbf{H}_1\mathbf{v}_K|\\
		& \geq \frac{1}{N} \left|\tr \left(\mathbf{H}_{K-1}\cdots\mathbf{H}_1\right)\right| -  \frac{2}{N} ~~~~({\rm by~ Cauchy's ~Inequality})
	\end{align*}
	and 
	\begin{align*}
		\frac{1}{N} |\tr\left(\mathbf{H}_1\right)| &= \frac{1}{N} |\tr\left(\mathbf{I}-2\mathbf{v}_1\mathbf{v}_1^H\right)|\\
		 &= 1 - \frac{2}{N},
	\end{align*}
	we have 
	\begin{equation}\label{a4}
		\frac{1}{N} \left|\tr \left(\mathbf{H}_K\cdots\mathbf{H}_1\right)\right| \geq 1 - \frac{2K}{N}.
	\end{equation}
	Hence,
	\begin{align*}
		\frac{1}{N} \sum_{j=1}^{N} \frac{|\mathbf{D}_{post, j}^H\mathbf{D}_{prior, j}|}{\vert|\mathbf{D}_{post, j}\vert|\vert|\mathbf{D}_{prior, j}\vert|} &=
		\frac{1}{N} \sum_{j=1}^{N} |\mathbf{D}_{post, j}^H\mathbf{D}_{prior, j}| ~~~~({\rm since~ \vert|\mathbf{D}_{post, j}\vert|=\vert|\mathbf{D}_{prior, j}\vert|=1})\\
		& \geq \frac{1}{N} \left|\sum_{j=1}^{N} \mathbf{D}_{post, j}^H\mathbf{D}_{prior, j}\right|\\
		& = \frac{1}{N} \left|\tr \left(\mathbf{D}_{post}^H\mathbf{D}_{prior}\right)\right|\\
		& = \frac{1}{N} \left|\tr \left(\left(\mathbf{H}_K\cdots\mathbf{H}_1\right)\mathbf{D}_{prior}^H\mathbf{D}_{prior}\right)\right| ~~~~({\rm by~ \eqref{HOT1}})\\
		& = \frac{1}{N} \left|\tr \left(\mathbf{H}_K\cdots\mathbf{H}_1\right)\right|\\
		& \geq 1 - \frac{2K}{N}. ~~~~({\rm by~ \eqref{a4}})
	\end{align*}
	Hence, 
	\begin{equation*}
		\rho(\mathbf{D}_{post}, \mathbf{D}_{prior}) \geq 1-\frac{2K}{N}.
	\end{equation*}
	The proof is complete.
\end{proof}

\begin{remark}
	Theorem \ref{generalization1} states that, in the scenario of multiple reference knowledge, as long as the number of reference knowledge components $K$ satisfies $K = o(N)$, the relative error between the posterior transform domain $\mathbf{D}_{post}$ and the prior transform domain $\mathbf{D}_{prior}$ asymptotically converges to 0, while their correlation approaches 1 as the dimension $N$ increases. This implies that when the number of reference knowledge components grows sublinearly with the dimension (i.e., $K = o(N)$), the posterior domain $\mathbf{D}_{post}$ not only represents a minimal correction of $\mathbf{D}_{prior}$ in the rank sense, but also remains nearly indistinguishable from $\mathbf{D}_{prior}$ under the metrics of relative error and correlation. Consequently, the prior and posterior transform domains retain similar generalization properties.
\end{remark}


We can similarly establish the specificity theory for the posterior transform domain in the multiple reference knowledge scenario.

Define the $K$-energy concentration of a vector $\mathbf{a} \in \mathbb{C}^N$ as the ratio of the squared $\ell_2$-norm of its $K$ largest-magnitude components to the total energy of the vector:  
\begin{equation*}
	\gamma_K(\mathbf{a}) = \frac{\sum_{i=1}^K |a_{(i)}|^2}{\|\mathbf{a}\|_2^2},
\end{equation*}
where $|a_{(1)}| \geq |a_{(2)}| \geq \cdots \geq |a_{(n)}|$ denote the components of $\mathbf{a}$ ordered by descending magnitude.
This metric quantifies how concentrated the energy of $\mathbf{a}$ is in a $K$-sparse vector, with $\gamma_K(\mathbf{a}) \in \left[\frac{K}{N}, 1\right]$. Higher values indicate greater sparsity.

The following theorem reveals the specificity inherent to the posterior transform domain $\mathbf{D}_{post}$ in HOT.

\begin{theorem}\label{specificity1} 
	
	Let $\mathbf{x}$ denote the true signal with sparse representations $\mathbf{w}_{prior}$ and $\mathbf{w}_{post}$ in the prior and posterior transform domains, respectively, i.e.,
	\begin{equation*}
		\mathbf{x} = \mathbf{D}_{prior}\mathbf{w}_{prior} = \mathbf{D}_{post}\mathbf{w}_{post},
	\end{equation*}
	and the representational fidelity of the true signal $\mathbf{x}$ with respect to the multi-reference knowledge $\mathbf{R}$ is defined as:
	\begin{equation*}
		\rho = \frac{\vert|\mathbf{R}\left(\mathbf{R}^H\mathbf{R}\right)^{-1}\mathbf{R}^H\mathbf{x}\vert|}{\vert|\mathbf{x}\vert|}
	\end{equation*}
	with $\rho \in \left[0, 1\right]$, where a larger $\rho$ indicates better representational fidelity of the true signal.
	
	(a) If $\rho \geq \sqrt{\gamma_K(\mathbf{w}_{prior})}$, then
	\begin{equation}\label{7a}
		\gamma_K(\mathbf{w}_{post}) \geq \gamma_K(\mathbf{w}_{prior}).
	\end{equation}
	
	(b) Let $\mathbf{w}_i = \mathbf{D}_{prior}^H \mathbf{r}_i, ~\forall i = 1,\dots,K$ denote the reference knowledge on the prior transform domain. As long as $\mathbf{w}_i$ captures partial support information of $\mathbf{w}_{prior}$, i.e., $\text{supp}(\mathbf{w}_i) \subseteq \text{supp}(\mathbf{w}_{prior})$, and $\mathbf{r}_i^H \mathbf{D}_{prior,j_i} \neq 0$, then
	\begin{equation}
		\vert|\mathbf{w}_{post}\vert|_0 \le \vert|\mathbf{w}_{prior}\vert|_0.
	\end{equation}
	
	(c) Suppose the the true signal is at least $K$-sparse on the prior transform domain, i.e., $\vert|\mathbf{w}_{prior}\vert|_0 \geq K$. Define prior sparsity odd \footnote{Also known as numerical sparsity in \cite{lopes2013estimating}.} as
	\begin{equation}
		odd = \frac{\vert|\mathbf{w}_{prior}\vert|_1}{\vert|\mathbf{w}_{prior}\vert|_2},
	\end{equation}
	where $\sqrt{K} \le odd \le \sqrt{N}$. We have
	\begin{equation}\label{7b}
		\vert|\mathbf{w}_{post}\vert|_1 \le \vert|\mathbf{w}_{prior}\vert|_1,
	\end{equation}
	when
	
	(i) $\sqrt{K} \le odd \le \sqrt{N-K}$: 
	\begin{equation}\label{7a1}
		\rho\geq \frac{\sqrt{K}odd+\sqrt{(N-K)(N-odd^2)}}{N}.
	\end{equation}
	
	(ii) $\sqrt{N-K} \le odd \le \sqrt{N}$:
	\begin{equation}\label{7a2}
		\rho\geq \frac{\sqrt{K}odd+\sqrt{(N-K)(N-odd^2)}}{N}~~or~~\rho\le \frac{\sqrt{K}odd-\sqrt{(N-K)(N-odd^2)}}{N}.
	\end{equation}
\end{theorem}

\begin{proof}

	We begin the proof with Case (a) in Theorem \ref{specificity1}. Denote $T = \{j_1,\dots,j_K\}$. By the construction of HOT with multiple reference knowledge in \eqref{multiHOT}, we have $\text{span}\{\mathbf{r}_1,\dots,\mathbf{r}_K\} = \text{span}\{\mathbf{D}_{post, j_1},\dots,\mathbf{D}_{post, j_K}\}$. Denote $\mathbf{w}_{post, T} \in \mathbb{C}^{N\times 1}$ such that $(\mathbf{w}_{post,T})_j = (\mathbf{w}_{post})_j$ if $j \in T$, and $(\mathbf{w}_{post,T})_j = 0$ if $j \in T^c$. Hence,
	\begin{equation}\label{a5}
		\vert|\mathbf{w}_{post,T}\vert|_2 = \vert|\mathbf{D}_{post}\mathbf{w}_{post,T}\vert|_2 = \vert|\mathbf{R}\left(\mathbf{R}^H\mathbf{R}\right)^{-1}\mathbf{R}^H\mathbf{x}\vert|_2 = \rho \vert|\mathbf{x}\vert|_2.
	\end{equation}
	Hence,
	\begin{equation*}
		\gamma_K(\mathbf{w}_{post}) \geq  \frac{\vert|\mathbf{w}_{post,T}\vert|_2^2}{\|\mathbf{w}_{post}\|_2^2} 
		= \frac{\rho^2\vert| \mathbf{x}\vert|_2^2}{\|\mathbf{w}_{post}\|_2^2}
		= \rho^2 
		\geq \gamma_K(\mathbf{w}_{prior}).
	\end{equation*}

	We now proceed to the proof of Case (b) in Theorem \ref{specificity1}. By \eqref{HOT1},
	\begin{align*}
		\mathbf{w}_{post} &= \mathbf{D}_{post}^H\mathbf{x}\\
		&= 
		\mathbf{H}_K\cdots\mathbf{H}_1\mathbf{D}_{prior}^H\mathbf{x}\\
		&= \mathbf{H}_K\cdots\mathbf{H}_1\mathbf{w}_{prior}\\
		&= \left(\mathbf{I}-2\mathbf{v}_K\mathbf{v}_K^H\right)\mathbf{H}_{K-1}\cdots\mathbf{H}_1\mathbf{w}_{prior}\\
		&= \mathbf{H}_{K-1}\cdots\mathbf{H}_1\mathbf{w}_{prior} - 2\left(\mathbf{v}_K^H\mathbf{H}_{K-1}\cdots\mathbf{H}_1\mathbf{w}_{prior}\right)\mathbf{v}_K \\
		&= \cdots \\
		&= \mathbf{w}_{prior} - \sum_{i=1}^{K} 2\left(\mathbf{v}_i^H\mathbf{H}_{i-1}\cdots\mathbf{H}_1\mathbf{w}_{prior}\right)\mathbf{v}_i,
	\end{align*}
	where
	\begin{equation}
		\mathbf{v}_i = \frac{(\mathbf{H}_{i-1}\cdots\mathbf{H}_1\mathbf{w}_i - \alpha_i\mathbf{e}_{j_i})_{\Omega \setminus \{j_1, j_2,\cdots j_{i-1}\}}}{\vert|(\mathbf{H}_{i-1}\cdots\mathbf{H}_1\mathbf{w}_i - \alpha_i\mathbf{e}_{j_i})_{\Omega \setminus \{j_1, j_2,\cdots j_{i-1}\}}\vert|_2}.
	\end{equation} 	
	Since $\mathbf{r}_{i}^H \mathbf{D}_{prior,j_{i}} = \mathbf{w}_{i}^H\mathbf{e}_{j_{i}}\neq 0$, $j_{i} \in \text{supp}(\mathbf{w}_{i}) \subset \text{supp}(\mathbf{w}_{prior})$. Now we prove that $\text{supp}(\mathbf{H}_{r}\cdots\mathbf{H}_1\mathbf{w}_{i}) \subset \text{supp}(\mathbf{w}_{prior})$ for any $r, i = 1,\dots,K$. 
	
	1. When $r=1$, 
	\begin{align*}
		\mathbf{H}_1\mathbf{w}_{i} &= \left(\mathbf{I}-2\mathbf{v}_1\mathbf{v}_1^H\right)\mathbf{w}_{i}\\
		&= \mathbf{w}_{i} - 2\left(\mathbf{v}_1^H\mathbf{w}_{i}\right)\mathbf{v}_1\\
		&= \mathbf{w}_{i} - 2\left(\mathbf{v}_1^H\mathbf{w}_{i}\right)\frac{\mathbf{w}_1 - \alpha_1\mathbf{e}_{j_1}}{\vert|\mathbf{w}_1 - \alpha_1\mathbf{e}_{j_1}\vert|_2},
	\end{align*}
	Since $\text{supp}(\mathbf{w}_{i}) \subset \text{supp}(\mathbf{w}_{prior})$, $\text{supp}(\mathbf{w}_{1}) \subset \text{supp}(\mathbf{w}_{prior})$, $j_1 \in \text{supp}(\mathbf{w}_{prior})$, $\text{supp}(\mathbf{H}_1\mathbf{w}_{i}) \subset \text{supp}(\mathbf{w}_{prior})$ for any $i = 1,\dots,K$.
	
	2. Suppose $\text{supp}(\mathbf{H}_{r-1}\cdots\mathbf{H}_1\mathbf{w}_{i}) \subset \text{supp}(\mathbf{w}_{prior})$ for any $i = 1,\dots,K$, we have
	\begin{align*}
		\mathbf{H}_r\cdots\mathbf{H}_1\mathbf{w}_{i} &= \left(\mathbf{I} - 2\mathbf{v}_r\mathbf{v}_r^H\right)\mathbf{H}_{r-1}\cdots\mathbf{H}_1\mathbf{w}_{i}\\
		&= \mathbf{H}_{r-1}\cdots\mathbf{H}_1\mathbf{w}_{i} - 2\left(\mathbf{v}_r^H\mathbf{H}_{r-1}\cdots\mathbf{H}_1\mathbf{w}_{i}\right)\mathbf{v}_r\\
		&= \mathbf{H}_{r-1}\cdots\mathbf{H}_1\mathbf{w}_{i} - 2\left(\mathbf{v}_r^H\mathbf{H}_{r-1}\cdots\mathbf{H}_1\mathbf{w}_{i}\right)\frac{(\mathbf{H}_{r-1}\cdots\mathbf{H}_1\mathbf{w}_r - \alpha_r\mathbf{e}_{j_r})_{\Omega \setminus \{j_1, j_2,\cdots j_{r-1}\}}}{\vert|(\mathbf{H}_{r-1}\cdots\mathbf{H}_1\mathbf{w}_r - \alpha_r\mathbf{e}_{j_r})_{\Omega \setminus \{j_1, j_2,\cdots j_{r-1}\}}\vert|_2}.
	\end{align*}
	Since $\text{supp}(\mathbf{H}_{r-1}\cdots\mathbf{H}_1\mathbf{w}_{i}) \subset \text{supp}(\mathbf{w}_{prior})$, $\text{supp}(\mathbf{H}_{r-1}\cdots\mathbf{H}_1\mathbf{w}_r) \subset \text{supp}(\mathbf{w}_{prior})$, $j_r \in \text{supp}(\mathbf{w}_{prior})$, $\text{supp}(\mathbf{H}_r\cdots\mathbf{H}_1\mathbf{w}_{i}) \subset \text{supp}(\mathbf{w}_{prior})$ for any $i = 1,\dots,K$.
	
	Hence, $\text{supp}(\mathbf{H}_{r}\cdots\mathbf{H}_1\mathbf{w}_{i}) \subset \text{supp}(\mathbf{w}_{prior})$ for any $r, i = 1,\dots,K$. 
	
	Hence, $\text{supp}(\mathbf{v}_i) \subset \text{supp}(\mathbf{w}_{prior})$ and 
	$\text{supp}(\mathbf{w}_{post}) \subset \text{supp}(\mathbf{w}_{prior})$. 
	
	Hence, 
	\begin{equation*}
		\vert|\mathbf{w}_{post}\vert|_0 \le \vert|\mathbf{w}_{prior}\vert|_0.
	\end{equation*}
	
	Finally, we establish the proof for Case (c) of Theorem \ref{specificity1}.
	\begin{align*}
		\vert|\mathbf{w}_{post}\vert|_1 &= \vert|\mathbf{w}_{post, T}\vert|_1 + \vert|\mathbf{w}_{post, T^c}\vert|_1\\
		&\le \sqrt{K}\rho \vert|\mathbf{x}\vert|_2 + \sqrt{\left(1-\rho^2\right)\left(N-K\right)}\vert|\mathbf{w}_{prior}\vert|_2 ~~~~({\rm by~ Cauchy's ~Inequality~ and ~\eqref{a5}})\\
		&= \left(\sqrt{K}\rho + \sqrt{\left(1-\rho^2\right)\left(N-K\right)}\right)\vert|\mathbf{w}_{prior}\vert|_2.
	\end{align*}
	
	When 
	\begin{equation}\label{equ1}
		\sqrt{K}\rho + \sqrt{\left(1-\rho^2\right)\left(N-K\right)} \le \frac{\vert|\mathbf{w}_{prior}\vert|_1}{\vert|\mathbf{w}_{prior}\vert|_2} = odd,
	\end{equation}
	we have $\vert|\mathbf{w}_{post}\vert|_1 \le \vert|\mathbf{w}_{prior}\vert|_1$. \eqref{equ1} is equivalent to 
	\begin{equation}\label{equ2}
		N\rho^2 -2~odd~\sqrt{K}\rho + \left(odd^2-N+K\right) \geq 0, ~~0\le \rho \le 1.
	\end{equation}
	Solving inequality \eqref{equ2} yields the condition \eqref{7a1} and \eqref{7a2} on $\rho$. The proof is complete.
\end{proof}

\begin{remark}
	Theorem \ref{specificity1} establishes that when the multiple reference knowledge $\mathbf{R}$ captures partial information about the true signal $\mathbf{x}$, the sparsity of $\mathbf{x}$ on the posterior transform domain $\mathbf{D}_{post}$—quantified by energy concentration, $\ell_0$-norm, and $\ell_1$-norm—outperforms its sparsity on the prior transform domain $\mathbf{D}_{prior}$. Theorem \ref{specificity} serves as a special case of Theorem \ref{specificity1} when $K=1$.
\end{remark}

\begin{remark}
	For the posterior transform domain $\mathbf{D}_{post}$ in the HOT framework with multiple reference knowledge, the conclusions in Theorem \ref{specificity} (a) and (c) regarding any individual reference knowledge component remain valid.
\end{remark}

\begin{remark}
	The index $\{j_1,\dots,j_K\}$ selected via Remark \ref{selection1} inherently satisfies one of the two conditions in Case (b) of Theorem \ref{specificity1}, specifically $\mathbf{r}_i^H \mathbf{D}_{prior,j_i} \neq 0$. This demonstrates that when applying Remark \ref{selection1}, as long as $\mathbf{w}_i$ captures partial support information of $\mathbf{w}_{prior}$, i.e., $\text{supp}(\mathbf{w}_i) \subseteq \text{supp}(\mathbf{w}_{prior})$, then
	\begin{equation}
		\vert|\mathbf{w}_{post}\vert|_0 \le \vert|\mathbf{w}_{prior}\vert|_0.
	\end{equation} 
\end{remark}

\begin{remark}
	The assumption on the true signal in the prior transform domain ($\vert|\mathbf{w}_{prior}\vert|_0 \geq K$ and $odd \geq \sqrt{K}$) stated in Theorem \ref{specificity1} is reasonable. This is because, when processing more than $K$ linearly independent target signals simultaneously, there must exist some signals among them whose $\ell_0$-norm in the prior transform domain exceeds $K$.

	In Theorem \ref{specificity1} (c), the parameter $odd$ quantifies the sparsity of the representation $\mathbf{w}_{prior}$ on the prior transform domain $\mathbf{D}_{prior}$. A larger $odd$ corresponds to a less sparse $\mathbf{w}_{prior}$.  
	
	(i) For $ \sqrt{K} \leq odd \leq \sqrt{N-K} $, a smaller $odd$ demands a higher representation fidelity $\rho $ of the true signal $\mathbf{x}$ with respect to the multiple reference knowledge $\mathbf{R}$. Specifically, when $odd = \sqrt{K}$, $\mathbf{w}_{prior}$ achieves best sparsity (sparsity of $K$), and representation fidelity $\rho $ must equal 1 to satisfy \eqref{7b}.  
	
	(ii) For $odd \geq \sqrt{N-K}$, even a small representation fidelity $\rho$ of the true signal $\mathbf{x}$ with respect to the multiple reference knowledge $\mathbf{R}$ suffices to satisfy \eqref{7b}. This indicates that the multiple reference knowledge $\mathbf{R}$ can provide meaningful information regardless of whether $\rho$ is large or small. 
	
	(iii) When $odd = \sqrt{N}$, the representation $\mathbf{w}_{post}$ on the posterior transform domain $\mathbf{D}_{post}$ is guaranteed to be sparser than $\mathbf{w}_{prior}$, irrespective of the choice of reference knowledge $\mathbf{R}$.  
\end{remark}

Using DWT as the prior, we construct HOT transforms with two reference vectors to demonstrate its performance across multiple images. The transformation results are shown in Figure \ref{fig5}, where the two references are the column-wise mean vectors of four images. Panels (a) and (b) display the heatmaps of the prior DWT domain and the posterior HOT domain, respectively. Panels (c), (d), and (e) show the original images, their representations in the DWT domain, and their counterparts in the HOT domain. Even with such coarse reference information, the HOT domain yields markedly better sparsity and energy compaction than DWT. Meanwhile, the differences between HOT and its prior DWT remain minimal—preserving the generalization capability of the prior. In essence, HOT achieves a minor adjustment yielding outsized benefits, delivering substantial representational gains through only slight, informed modifications to a classical transform.

\begin{figure}[h]
	\centering
	\includegraphics[width=1\linewidth]{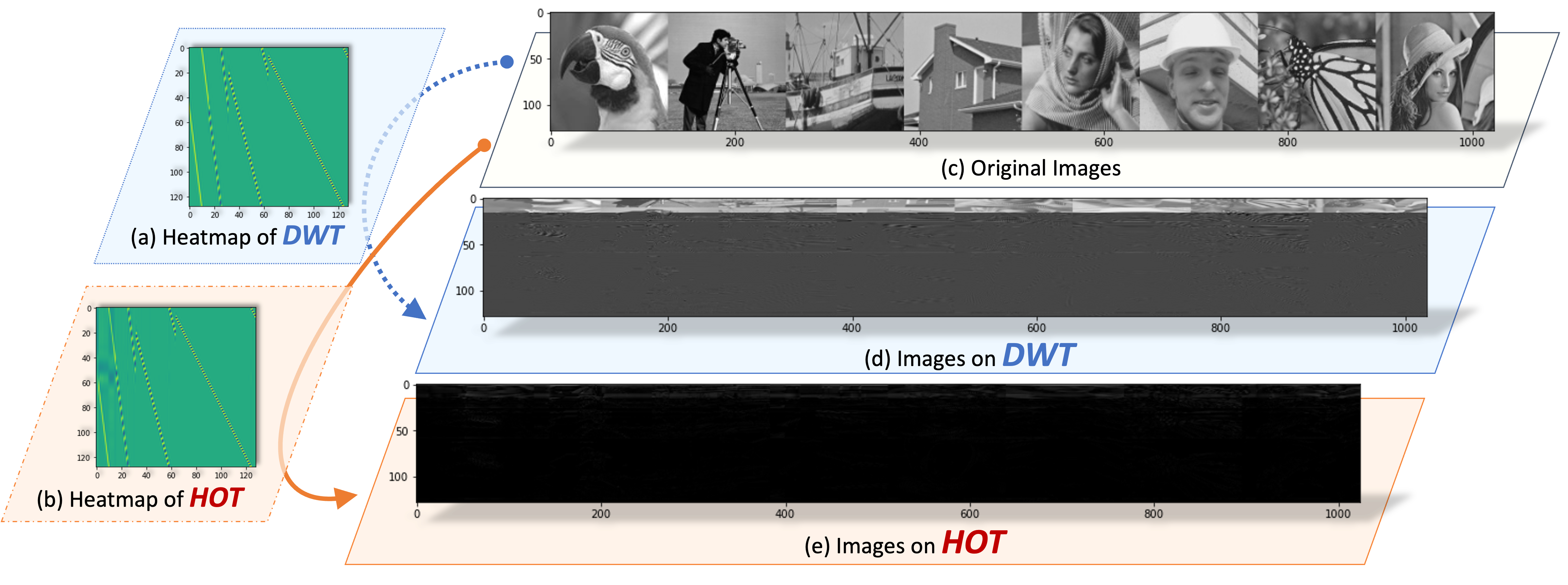}
	\vspace{-2mm}
	\caption{HOT with two reference knowledge (each is average of columns in 4 images).}
	\label{fig5}
\end{figure}

\section{Experiment}

In this section, we conduct extensive evaluations of HOT's effectiveness across diverse tasks, multimodal data, and different algorithms. A natural question arises: what types of reference information can be practically obtained in real-world compressed sensing scenarios? Here we provide three representative examples: (1) a solution generated by a weak CS learner, (2) the solution from previous steps in sequential CS tasks, and (3) an initial coarse estimate of the target. We subsequently demonstrate how these easily accessible references can be leveraged to construct HOT transforms in respective applications.

Specifically, we present results across three distinct real-world tasks—(1) audio sensing, (2) 5G time-varying channel estimation, and (3) image compression—spanning three different data modalities: (1) audio, (2) wireless channel, and (3) image, under three representative compressed sensing algorithms: (1) Orthogonal Matching Pursuit (OMP), (2) Basis Pursuit (BP), and (3) Least Absolute Shrinkage and Selection Operator (LASSO). Across all these tasks, data modalities and algorithms, HOT consistently delivers substantial meta-gains, underscoring its broad applicability and robustness in compressed sensing.

\subsection{Audio Sensing with HOT}
In this subsection, we demonstrate how to construct HOT transform using solution generated by a weak compressed sensing (CS) learner for audio sensing tasks, leading to tangible performance gains. The rationale for using weak CS learners, which are computationally cheap and provide coarse target estimates, lies in their ability to offer reference information at low cost. Despite its potential inaccuracy, this coarse solution can be effectively utilized by the HOT transform to produce significant improvements, as previously discussed. Here, We propose two practical schemes:
\begin{itemize}
	\item \textbf{Boosting a Weak CS Learner}: A HOT transform is constructed from solution of a weak CS learner. By repeatedly applying the weak CS algorithm on the HOT domain and constructing new HOT transform, the original weak learner is progressively enhanced.
	\item \textbf{Enhancing a Stronger Solver with Weak Guidance}: Solutions from one weak CS learner are used to build the HOT transform, which is then leveraged by a different and potentially stronger CS solver to elevate its performance ceiling.
\end{itemize}

For the audio sensing task, we employ the classical Discrete Cosine Transform (DCT) as the prior transform domain for HOT. Simultaneously, the classic Orthogonal Matching Pursuit (OMP) algorithm, which is constrained to identify only 3-4 basis elements, serves as the weak CS learner. In the experiments, audio clips of varying lengths ($N$ ranging from 500 to 700) were randomly selected from dataset \cite{gemmeke2017audio}, while the number of observations $M$ was fixed at 90. The experimental results are presented in Figures \ref{fig6} and \ref{fig7}.
\begin{figure}[h]
	\centering
	\includegraphics[width=1\linewidth]{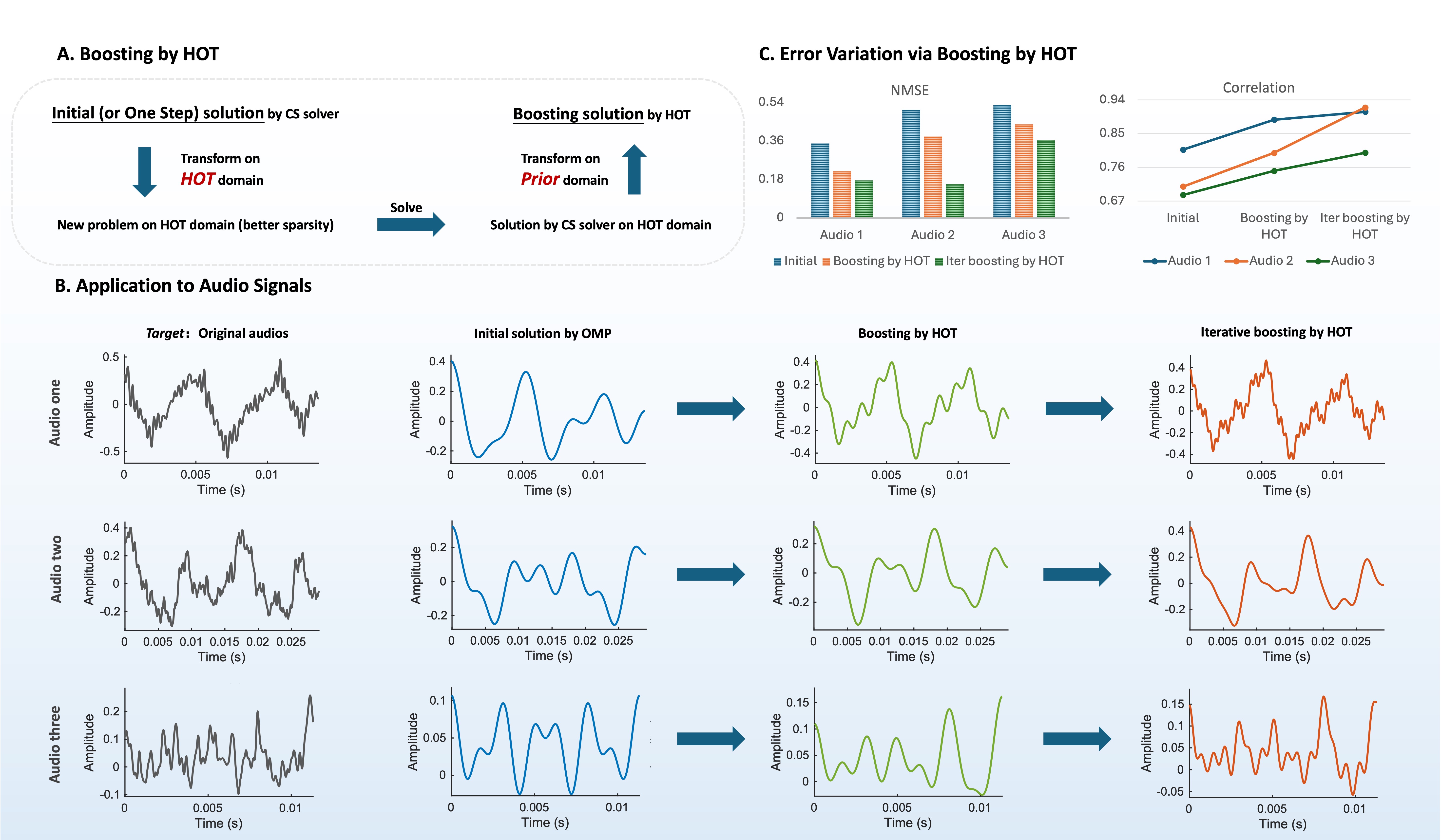}
	\vspace{-2mm}
	\caption{Boosting Weak Compressed Sensing Learners with HOT}
	\label{fig6}
\end{figure}

As presented in Figure \ref{fig6}, panel A outlines the framework for boosting a weak compressed sensing learner with HOT. Panel B displays the audio signals reconstructed by the weak CS learner after each round of HOT-based boosting, while panel C shows the corresponding NMSE ($\vert|\hat{\mathbf{x}}-\mathbf{x}^*\vert|^2/\vert|\mathbf{x}^*\vert|^2$) and correlation ($|\hat{\mathbf{x}}^H\mathbf{x}^*|/\vert|\hat{\mathbf{x}}\vert|\vert|\mathbf{x}^*\vert|$) evolution between the solutions generated by the boosted learner and the ground-truth audio. The figure reveals the evolutionary improvement of the weak CS learner by HOT boosting. Starting from an initial state of capturing only broad trends, the weak CS learner, benefiting from the enhanced sparsity and energy compaction in the HOT domain compared to the DCT domain, progresses to accurately reconstruct major audio components by the second round. As the reference information becomes more precise, the third round yields recovery of fine-grained details. After HOT boosting, the final recovery error of the weak CS learner achieves 50\% to 70\% reduction compared to its initial level.

\begin{figure}[h]
	\centering
	\includegraphics[width=1\linewidth]{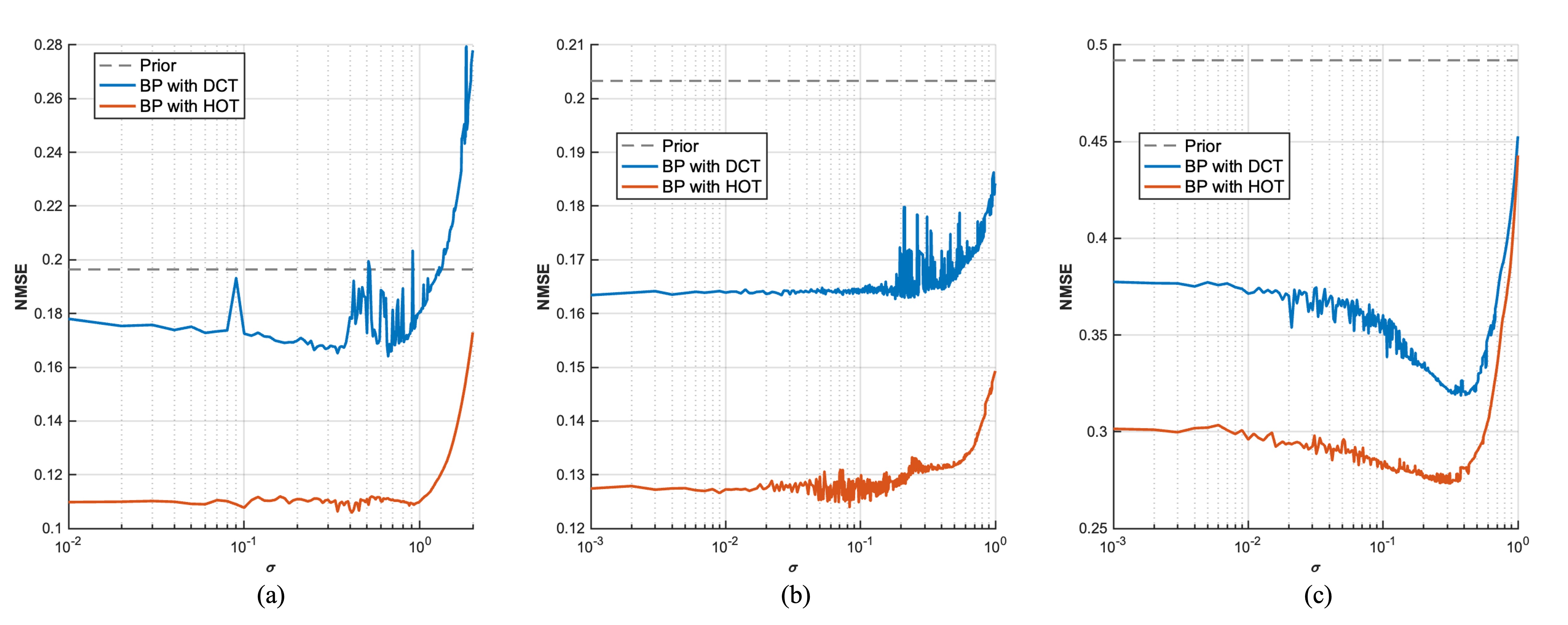}
	\vspace{-3mm}
	\caption{Enhancing a Stronger Solver with Weak Guidance}
	\label{fig7}
\end{figure}

As shown in Figure \ref{fig7}, we again employ the restricted OMP algorithm as the weak CS learner to construct the HOT transform. The Basis Pursuit (BP) algorithm is then applied for signal recovery in both the HOT and DCT domains. The dashed line (prior) indicates the relative error between the coarse solution from OMP and the ground-truth audio. The red and blue lines represent the recovery errors of BP in the HOT and DCT domains, respectively, plotted against varying hyperparameters of the BP algorithm. The results demonstrate that although the OMP-derived prior is less accurate than BP's direct recovery in the DCT domain, the HOT transform constructed from this reference still substantially elevates the performance ceiling of BP. This enhancement, which is robust to hyperparameter choices and improves both optimal and overall performance, underscores the efficacy of the HOT transform: even a reference that is inferior to the solver's native solution can provide a valuable complementary perspective, thereby pushing the performance limit of the algorithm.

\subsection{5G Channel Estimation with HOT}
In this subsection, we demonstrate the significant performance gains achieved by constructing HOT transforms from historical solutions in sequential compressed sensing task, using 5G time-varying channel estimation as a case study. Channel estimation serves as the cornerstone of modern mobile communications. Due to constrained communication resources, only partial channel observations can be acquired at each time step, necessitating the use of compressed sensing techniques to recover full-dimensional channel state information from limited measurements. We consider the CDL-B channel model from the 3GPP standard, which comprises 23 cluster paths and exhibits rapid time variation\cite{3gpp38901}. As one of the most complicated channel profiles in 5G, CDL-B poses considerable challenges for accurate channel estimation.

For the channel estimation task, we employ the classical Discrete Fourier Transform (DFT) as the prior transform domain for HOT, and capitalizes on the channel estimate from the previous time step as reference information, despite its potential inaccuracy from prior estimation errors and channel dynamics. We show that HOT consistently delivers substantial meta-gains under diverse observation ratios and noise conditions, even when leveraging such imperfect information. The experimental setup is structured as follows:
\begin{itemize}
	\item \textbf{Estimation Target:} The CDL-B time-varying channel is considered, with a focus on a 64-antenna system ($N=64$) and a sequence of 200 consecutive channel estimation instances ($T=200$).
	\item \textbf{Baseline Algorithm:} The Basis Pursuit (BP) method, a simple and fast compressed sensing algorithm, is employed as the baseline solver.
	\item \textbf{Measurement Ratio per Instance:} Comprehensive tests are conducted for a range of measurement ratios ($M/N$) from 0.25 to 0.75.
	\item \textbf{Signal-to-Noise Ratio (SNR):} The performance is systematically evaluated under various SNR conditions, ranging from 15 dB to 40 dB.
	\item \textbf{Evaluation Metrics:} We compare the performance of the BP algorithm operating in the DFT domain versus the HOT domain using three key metrics: TNMSE (Temporal Normalized Mean Square Error), TCorr (Temporal Correlation), and the total computational CPU time. The definitions of TNMSE and TCorr are given by:
	\begin{equation*}
		\text{TNMSE} = \frac{1}{T} \sum_{t=1}^{T} \frac{\vert|\hat{\mathbf{x}}_t-\mathbf{x}_t^*\vert|^2}{\vert|\mathbf{x}_t^*\vert|^2},\quad \text{TCorr} = \frac{1}{T} \sum_{t=1}^{T}\frac{|\hat{\mathbf{x}}_t^H\mathbf{x}_t^*|}{\vert|\hat{\mathbf{x}}_t\vert|\vert|\mathbf{x}_t^*\vert|}. 
	\end{equation*}
\end{itemize}
The experimental results are presented in Figure \ref{fig8}.
\begin{figure}[h]
	\centering
	\includegraphics[width=1\linewidth]{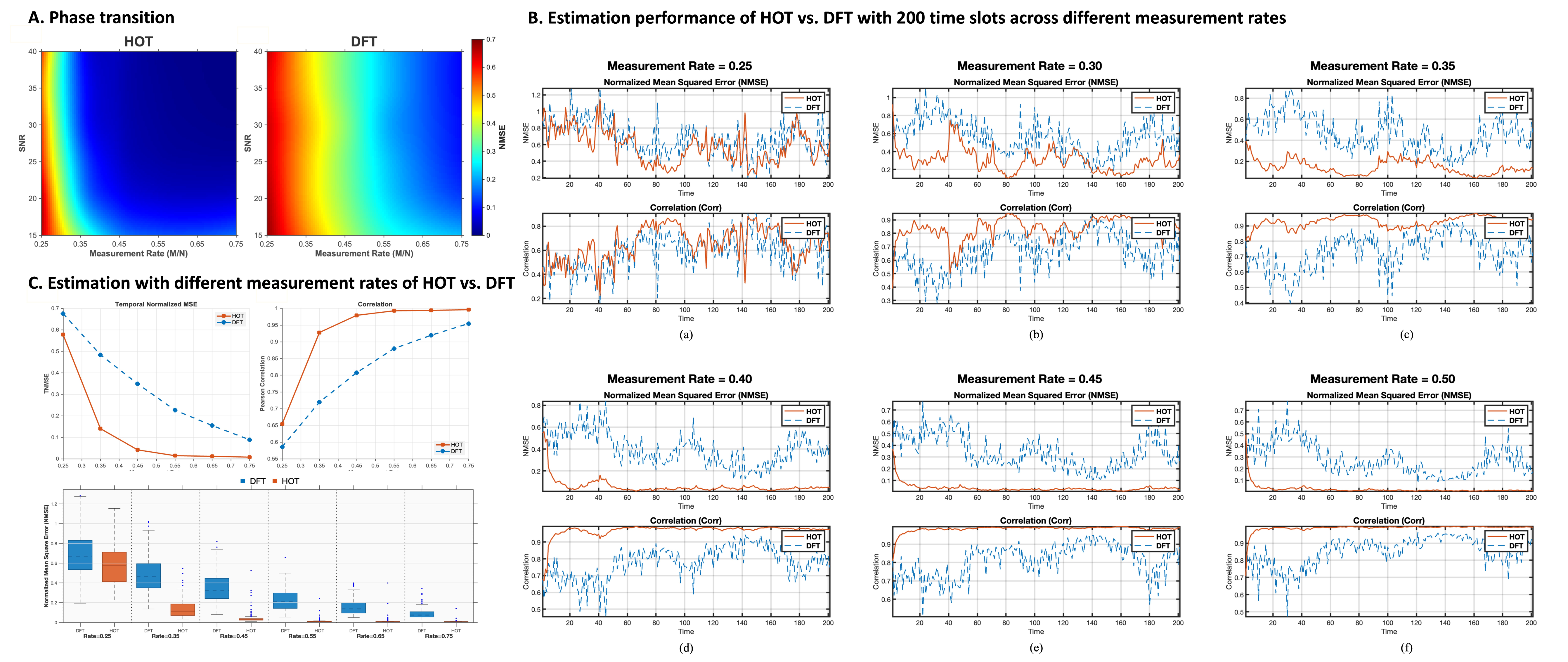}
	\vspace{-3mm}
	\caption{5G channel estimation with HOT}
	\label{fig8}
\end{figure}

Panel A of Figure \ref{fig8} presents the phase transition diagrams of the BP algorithm on the HOT and DFT domains under different measurement ratios and noise levels. Darker shades indicate better channel estimation performance under the corresponding measurement ratio and noise level, hence a larger blue region in the phase transition diagram reflects a stronger phase transition capability of the algorithm. As can be observed, using the same BP algorithm, the phase transition region on the HOT domain is three times larger than that on the DFT domain, elevating the solver from a naive baseline to a state-of-the-art level. Panel B further details, as the measurement ratio varies from 0.25 to 0.5, the temporal evolution of the relative error and correlation of channel recovery by the BP algorithm on the HOT and DFT domains at an SNR of 30 dB. In the extreme case with a measurement ratio of 0.25 ($M=16$), the relative error of the BP algorithm on the DFT domain remains around 0.7–0.8 and occasionally exceeds 1.0. Considering the inherent time-varying nature of the channel, the channel estimate from the previous time step serves as a highly inaccurate reference for the current channel state. Nevertheless, even under such conditions, the HOT transform constructed from this reference still provides a noticeable gain in channel estimation. This stems from the excellent generalization and specificity of HOT: when the reference information is highly inaccurate, HOT maintains sparsity after transformation due to its generalization ability comparable to that of DFT; whereas when the reference contains even a small amount of useful information, HOT exhibits superior specificity, thereby enhancing the recovery performance of compressed sensing. At a measurement ratio of 0.3 ($M=19$), the recovery error of the BP algorithm on the DFT domain remains around 0.5–0.8. Although the reference information is still coarse at this point, the BP algorithm in the HOT domain achieves a 70\% improvement compared to its performance on the DFT domain. As the measurement ratio increases further and the reference information becomes more accurate, this meta-gain can exceed 90\%. Panel C and Table \ref{tab1} further summarize the TNMSE and TCorr of channel estimation by the BP algorithm on the HOT and DFT domains under different measurement ratios.
\begin{table}[htbp]
	\centering
	\caption{Comparison on TNMSE, TCorr and CPU time between HOT and DFT}
	\label{tab1}
	\begin{tabular}{ccccc}
		\toprule
		Measurement Rate & Transform Domain & ~~~~~TNMSE~~~~~ & ~~~~~TCorr~~~~~ & ~~~~CPU\ time~~~~ \\
		\midrule
		\multirow{2}{*}{$M/N = 0.25$} & DFT & 0.6756 & 0.5857 & 0.6037 \\
		& HOT & \textbf{0.5776} & \textbf{0.6540} & \textbf{0.5324} \\
		\cline{2-5}
		\multirow{2}{*}{$M/N = 0.35$} & DFT & 0.4842 & 0.7196 & 0.7779 \\
		& HOT & \textbf{0.1405} & \textbf{0.9280} & \textbf{0.6646} \\
		\cline{2-5}
		\multirow{2}{*}{$M/N = 0.45$} & DFT & 0.3488 & 0.8079 & 0.9567 \\
		& HOT & \textbf{0.0416} & \textbf{0.9793} & \textbf{0.8371} \\
		\cline{2-5}
		\multirow{2}{*}{$M/N = 0.55$} & DFT & 0.2264 & 0.8798 & 1.1736 \\
		& HOT & \textbf{0.0146} & \textbf{0.9928} & \textbf{1.0550} \\
		\cline{2-5}
		\multirow{2}{*}{$M/N = 0.65$} & DFT & 0.1548 & 0.9198 & 1.3751 \\
		& HOT & \textbf{0.0116} & \textbf{0.9941} & \textbf{1.2319} \\
		\cline{2-5}
		\multirow{2}{*}{$M/N = 0.75$} & DFT & 0.0890 & 0.9546 & 1.6156 \\
		& HOT & \textbf{0.0078} & \textbf{0.9961} & \textbf{1.4708} \\
		\bottomrule
	\end{tabular}
\end{table}

The third column of Table \ref{tab1} records the execution time of the BP algorithm for 200 instances of channel estimation on both the HOT and DFT domains. The results indicate that the BP algorithm not only achieves significantly superior estimation performance on the HOT domain but also operates approximately 10\% faster than on the DFT domain. This improvement can be attributed to the enhanced sparsity of the channel representation on the HOT domain. Since compressed sensing converges faster on sparser targets, the BP algorithm requires fewer computational steps to reach a solution on the HOT domain. Consequently, the proposed approach demonstrates simultaneous improvements in both estimation accuracy and computational efficiency compared to the conventional DFT-based method.

\subsection{Image Compression with HOT}
In this subsection, we demonstrate how to construct HOT transform using an initial coarse estimate of the target to deliver significant meta-gains for image compression tasks. Two common schemes are widely adopted in image compression:

\begin{itemize}
	\item \textbf{Sparse Transform-Based Compression:} The image is first transformed into a sparse domain (e.g. DWT), where only a small number of nonzero components with high energy are retained and the rest are discarded. This enables efficient storage of the image. During reconstruction, the preserved sparse coefficients are mapped back to the image domain. A higher number of retained components generally leads to better reconstruction quality, albeit at the cost of increased storage.
	\item \textbf{Compressed Sensing-Based Reconstruction:} This approach directly utilizes a sensing matrix to acquire compressed measurements of the image. Reconstruction is then performed using CS recovery algorithms that leverage the inherent transform sparsity of the image.
\end{itemize}

\vspace{-2mm}
For the image compression task, we employ the classical Discrete Wavelet Transform (DWT), the foundation of the JPEG-2000 standard, as the prior transform domain for HOT. Evaluations are conducted on eight classic $128 \times 128$ grayscale images sourced from two publicly available datasets \footnote{Available at \url{http://dsp.rice.edu/software/DAMP-toolbox} and \url{http://see.xidian.edu.cn/faculty/wsdong/NLR_Exps.htm}.}. The column-wise mean across all images (a $128 \times 1$ vector) is used as coarse reference information. While this mean vector provides only a rough approximation of any individual image, we demonstrate that the constructed HOT transform consistently outperforms conventional DWT in both of the previously mentioned compression schemes. Experimental results are summarized in Figure \ref{fig9}.

\begin{figure}[h]
	\centering
	\includegraphics[width=1\linewidth]{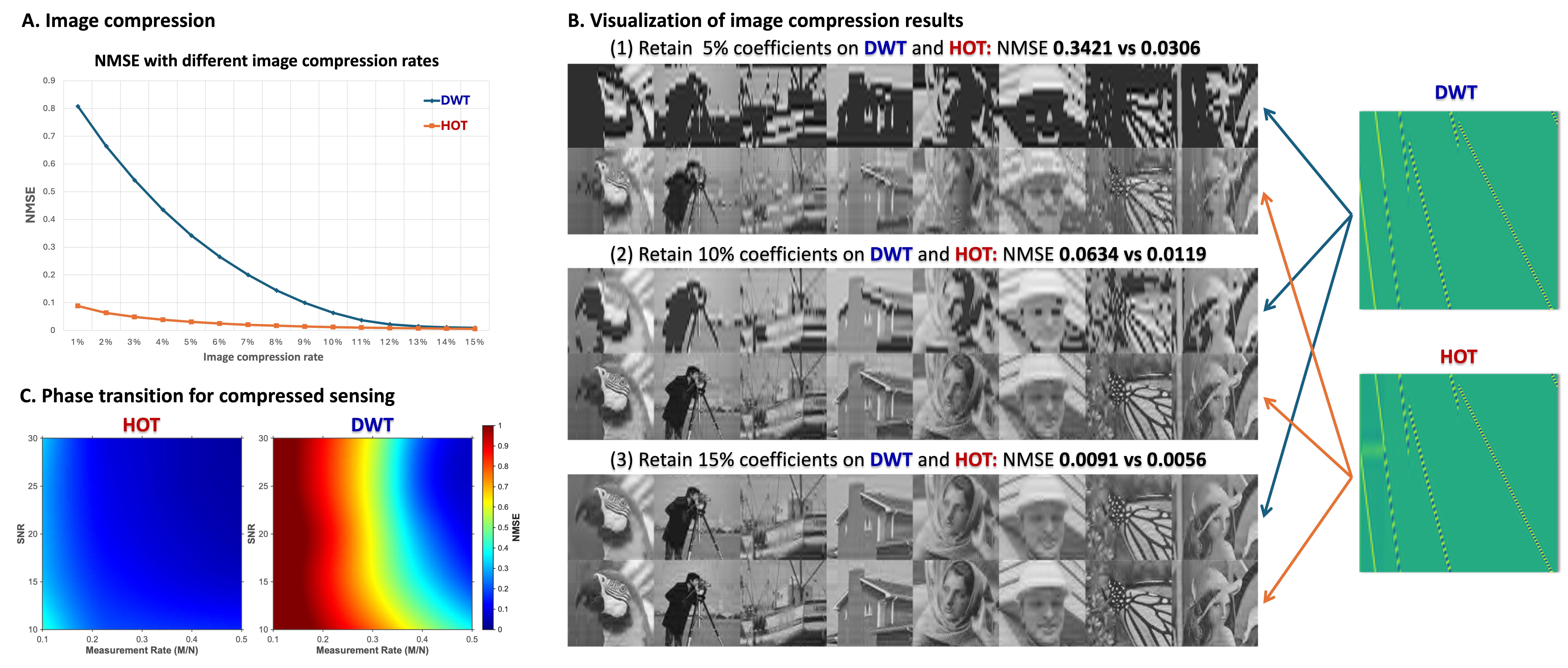}
	\vspace{-3mm}
	\caption{Image compression with HOT}
	\label{fig9}
\end{figure}

In Figure \ref{fig9}, Panels A and B illustrate the performance under the sparse transform-based compression scheme, while Panel C corresponds to the compressed sensing-based reconstruction scenario. Panel A displays the overall NMSE trend for image reconstruction when preserving only the top 1\% to 15\% of the highest-energy components in the HOT and DWT domains, respectively, across eight image sets. The results clearly indicate that for the same fraction of retained coefficients, the reconstruction error on the HOT domain is significantly lower than that on the DWT domain. Notably, when preserving up to 5\% of the components, the improvement exceeds 90\%. This implies that images can be compressed and stored at substantially lower cost on the HOT domain compared to the DWT domain. Panel B further provides visual comparisons of the reconstructed images when retaining 5\%, 10\%, and 15\% of the coefficients in each domain. Under the extreme compression scenario (5\% retention), the DWT reconstruction only captures rough outlines, whereas HOT already recovers most image content coherently. With 10\% retention, the visual quality achieved by DWT is slightly inferior to that of HOT with only 5\% retention, while HOT at 10\% yields clearly recognizable images. At 15\% retention, DWT produces results similar to HOT at 10\%, whereas HOT further enhances fine details such as facial features, background texture, and stripe patterns. The heatmaps on the right highlight that the difference between DWT and HOT domain is visually subtle, demonstrating that HOT achieves remarkable compression gains with minimal representational overhead. Panel C reports the gains of HOT in compressed sensing-based reconstruction. We randomly sampled from the eight image sets and conducted extensive reconstruction experiments under various measurement ratios and SNR levels. The experimental configuration is as follows:
\vspace{-2mm}
\begin{itemize}
	\item \textbf{Measurement Ratio:} The measurement ratio $M/N$ for image compressed sensing varies from 0.1 to 0.5.
	\item \textbf{Signal-to-Noise Ratio (SNR):} The SNR in the experiments ranges from 10 dB to 30 dB.
	\item \textbf{Baseline Algorithm:} We employ the classic compressed sensing method, LASSO, as the baseline algorithm for image reconstruction, evaluating its phase transition capabilities on both the HOT and DWT domains. The resulting phase transition diagram is shown in panel C of Figure \ref{fig9}.
\end{itemize}
\vspace{-2mm}
As observed in the figure, the same LASSO method achieves a phase transition region in the HOT domain that is four times larger than that in the DWT domain for image compressed sensing tasks. This further validates the effectiveness of HOT for compressed sensing-based reconstruction schemes.

Figure \ref{fig10} further demonstrates the critical importance of the prior transform domain. It compares image reconstruction results obtained by retaining the top 15\% of highest-energy components across three different transform domains: the Householder transform (HT) constructed from coarse reference information (equivalent to using the identity matrix as the prior transform domain), the conventional DWT, and HOT (with DWT as the prior).
\begin{figure}[h]
	\centering
	\includegraphics[width=1\linewidth]{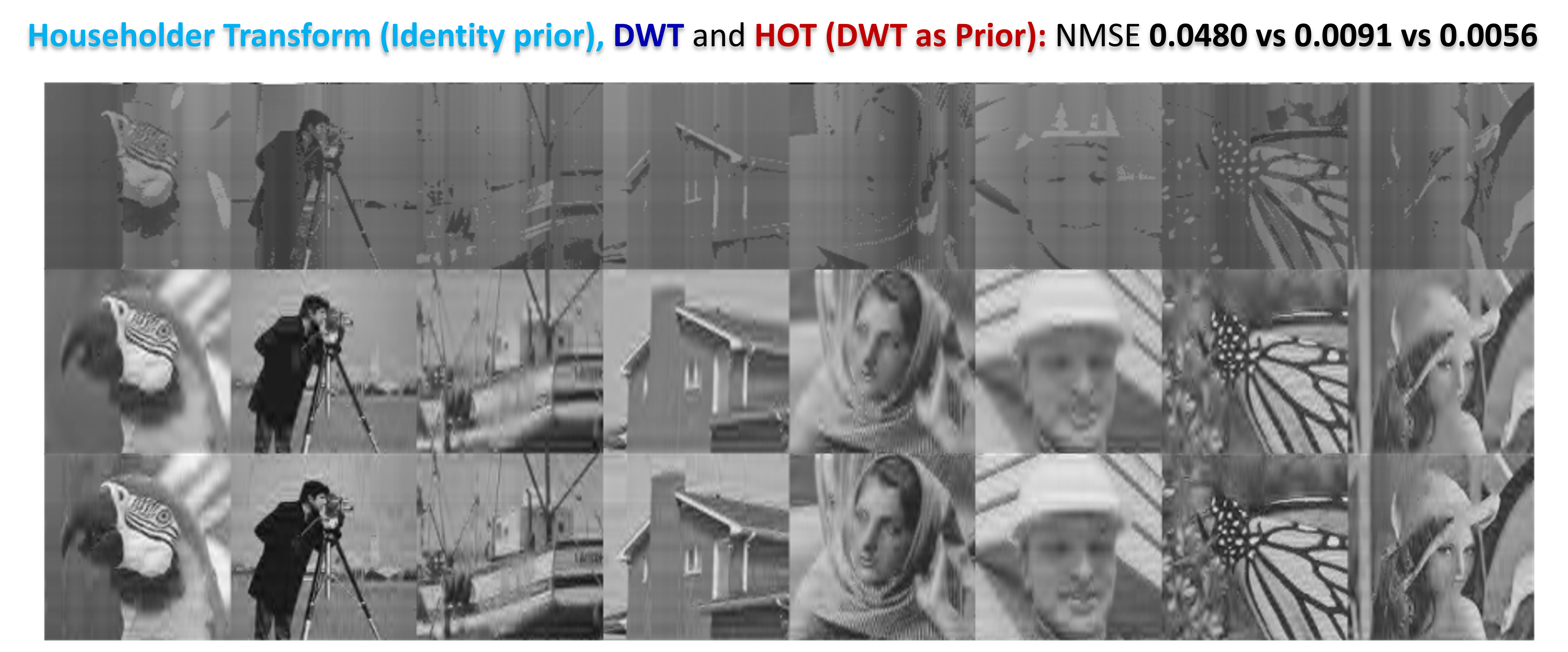}
	\vspace{-5mm}
	\caption{Importance of the Prior Transform Domain in HOT}
	\label{fig10}
\end{figure}

From the first row of Figure \ref{fig10}, it can be observed that due to its generalization properties resembling those of the identity matrix, HT struggles to achieve effective image compression when the reference information is inaccurate: the reconstructed images retain only partial content. In contrast, the third row of Figure \ref{fig10} illustrates that HOT effectively inherits the strong generalization capability of DWT while, thanks to the knowledge provided by the reference information, exhibits superior specificity, leading to significantly better performance in reconstructing fine image details. These findings further validate the theoretical framework established earlier.

\section{Conclusion}
This paper presented the Prior-to-Posterior Sparse Transform (POST) framework, a new paradigm for sparse representation that effectively resolves the long-standing trade-off between generalization and specificity in compressed sensing. Through systematic integration of any existing transform domains with task-specific reference knowledge, POST enables adaptive signal representation across diverse scenarios. The derived HOT transform demonstrates robust performance for both real and complex-valued signals, with theoretical guarantees under both single and multiple reference settings. Crucially, HOT maintains strong generalization while achieving significantly enhanced specificity even under limited reference accuracy.

Extensive experimental validation across audio sensing, 5G channel estimation, and image compression tasks confirms that HOT delivers consistent meta-gains for multiple reconstruction algorithms in multimodal scenarios, all with negligible computational overhead and even less computation time. These findings position the POST framework and HOT transform as a versatile and efficient solution for advanced compressed sensing and a broader class of transform-dependent tasks. Future works include exploration on alternative formulations and objective functions within the POST framework, as well as extensions to potential domains such as tensor-based signal processing and machine learning.



\bibliography{scibib}

\bibliographystyle{IEEEtran}

\section*{Acknowledgments}
This research was supported by National Key R\&D Program of China under grant 2021YFA1003303.



\end{document}